\documentclass[12pt]{amsart}
\usepackage[colorlinks=true,citecolor=blue,linkcolor=blue]{hyperref}
\usepackage{latexsym,amsmath,amssymb}
\usepackage{accents}
\usepackage{enumerate}
\usepackage{xargs,cleveref}

\usepackage[textsize=small]{todonotes}
\usepackage[dvips,bottom=1.4in,right=1in,top=1in, left=1in]{geometry}

\usepackage{cleveref}

plus 6pt minus 12pt
\def\eps{\varepsilon}

\newcommandx{\yaHelper}[2][1=\empty]{%
\ifthenelse{\equal{#1}{\empty}}%
  { \ensuremath{ \scriptstyle{ #2 } } } 
  { \raisebox{ #1 }[0pt][0pt]{ \ensuremath{ \scriptstyle{ #2 } } } }  
}   

\newcommandx{\yrightarrow}[4][1=\empty, 2=\empty, 4=\empty, usedefault=@]{%
  \ifthenelse{\equal{#2}{\empty}}
  { \xrightarrow{ \protect{ \yaHelper[ #4 ]{ #3 } } } } 
  { \xrightarrow[ \protect{ \yaHelper[ #2 ]{ #1 } } ]{ \protect{ \yaHelper[ #4 ]{ #3 } } } } 
}

\newcommand{\Mosco}{\hspace{-.05cm}\yrightarrow{\scriptscriptstyle \mathrm{M}}[-1pt]\hspace{-.05cm}}

\def\N{{\mathbb N}}

\newtheorem{theorem}{Theorem}
\newtheorem{lemma}[theorem]{Lemma}
\newtheorem{corollary}[theorem]{Corollary}
\newtheorem{proposition}[theorem]{Proposition}

\theoremstyle{definition}

\newtheorem{definition}[theorem]{Definition}
\newtheorem{example}[theorem]{Example}

\def\dist{{\rm dist\,}}

\def\supp{{\rm supp\,}}

\newcommand{\R}{\mathbb{R}}

\newcommand{\brac}[1]{\left (#1 \right )}

\newcommand{\subsubset}{\subset\subset}

\newcommand{\barint}{
\rule[.036in]{.12in}{.009in}\kern-.16in \displaystyle\int }

\newcommand{\barcal}{\mbox{$ \rule[.036in]{.11in}{.007in}\kern-.128in\int $}}

\def\mvint_#1{\mathchoice
          {\mathop{\vrule width 6pt height 3 pt depth -2.5pt
                  \kern -8pt \intop}\nolimits_{\kern -3pt #1}}%
          {\mathop{\vrule width 5pt height 3 pt depth -2.6pt
                  \kern -6pt \intop}\nolimits_{#1}}%
          {\mathop{\vrule width 5pt height 3 pt depth -2.6pt
                  \kern -6pt \intop}\nolimits_{#1}}%
          {\mathop{\vrule width 5pt height 3 pt depth -2.6pt
                  \kern -6pt \intop}\nolimits_{#1}}}

\numberwithin{theorem}{section} \numberwithin{equation}{section}

\newcommand{\lap}{\Delta }
\newcommand{\aleq}{\precsim}
\newcommand{\ageq}{\succsim}
\newcommand{\aeq}{\approx}

\newcommand{\laps}[1]{(-\lap)^{\frac{#1}{2}}}

\newcommand{\lapms}[1]{I^{#1}}

\title{On a Fractional Version of a Murat Compactness Result and Applications}

\author{Harbir Antil}
\address[Harbir Antil]{Department of Mathematical Sciences and the Center for Mathematics and Artificial Intelligence (CMAI) 
George Mason University, Fairfax, VA 22030, USA. \newline
ORCiD: 0000-0002-6641-1449}
\email{hantil@gmu.edu}

\author{Carlos N. Rautenberg}
\address[Carlos N. Rautenberg]{Department of Mathematical Sciences and the Center for Mathematics and Artificial Intelligence (CMAI) 
George Mason University, Fairfax, VA 22030, USA.  \newline 
ORCiD: 0000-0001-9497-9296}
\email{crautenb@gmu.edu}

\author{Armin Schikorra}
\address[Armin Schikorra]{Department of Mathematics, University of Pittsburgh, Pittsburgh, PA 15261, USA.\newline
ORCiD: 0000-0001-9242-1782}
\email{armin@pitt.edu}



\begin{document}
\maketitle

\begin{abstract}
The paper provides an extension, to fractional order Sobolev spaces, of the classical result of 
Murat and Brezis which states that the {cone of positive elements} in $H^{-1}(\Omega)$ compactly 
embeds in $W^{-1,q}(\Omega)$, for every $q < 2$ and for any open and bounded set $\Omega$ with Lipschitz boundary. 
In particular, our proof contains the classical result. Several new analysis tools are developed 
during the course of the proof to our main result which are of wider interest. Subsequently, we 
apply our results to the convergence of convex sets and establish a fractional version of the Mosco 
convergence result of Boccardo and Murat. {We conclude with an application of this result 
to quasi-variational inequalities.}
\end{abstract}


\section{Introduction and main results}

Let $\Omega \subset \R^N$ be an open  bounded set with Lipschitz boundary and let $\partial\Omega$ denote its boundary. For domains with $\partial \Omega$ of Lipschitz type, Murat \cite{FMurat_1981a} initially proved  that if a sequence of non-negative distributions converges weakly in $H^{-1}(\Omega)$ then this convergence is strong in $W^{-1,q}(\Omega)$ every $q < 2$. A direct and succinct  proof of this result was later given by Brezis in \cite{B81} without any assumption on the regularity of $\partial\Omega$. This result has found many applications in the literature, in particular, we mention the result on convergence of convex sets in \cite{BM79} by Boccardo/Murat: The set of elements $v\in H_0^1(\Omega)$ such that $v\geq \psi_n$ a.e. in $\Omega$ converges in the sense of Mosco to the set of elements (also in $H_0^1(\Omega)$) $v\geq \psi$ provided that $\psi_n\rightharpoonup \psi$ in $W^{1,p}(\Omega)$ where $p > 2$. The latter is a consequence of the Murat/Brezis results and determines a variety of stability results of variational inequalities and optimization problems.

The first goal of this paper is to prove an analogue of Murat/Brezis's result, also without any assumption on the boundary regularity, for fractional order Sobolev spaces. Subsequently, the second goal is to establish the fractional version of the Boccardo/Murat convergence result for closed and convex sets of unilateral type. We emphasize that, the fractional operators have recently found several realistic applications in geophysics and imaging science, see for instance \cite{weiss2020fractional,HAntil_SBartels_2017a,HAntil_ZDi_RKhatri_2020a}. 
{For $s \in (0,1]$ and $p \in (1,\infty)$ denote by $H^{s,p}_{00}(\Omega)$ the fractional Sobolev space $H^{s,p}(\R^n)$ restricted to functions which are zero in $\R^N \backslash \Omega$; for the precise definition see \eqref{eq:Hsp00}.} Our main result can be stated as follows:

\begin{theorem}\label{th:convergence}
Let $\Omega \subset \R^N$ be any open bounded set with Lipschitz boundary.
Assume that $h_n,h \in H_{00}^{s,2}(\Omega)^\ast$, $s \in (0,1]$ such that $h_n\geq 0$ for $n\in\mathbb{N}$ (in the sense of distributions)  and $h_n\rightharpoonup h$ in $H_{00}^{s,2}(\Omega)^\ast$ as $n\to \infty$. Then,
\begin{equation*}
	h_n \to h \qquad \text{in}\quad H^{-s,q}(\Omega) := H^{s,q'}_{00}(\Omega)^{\ast},
\end{equation*}
as $n\to\infty$, for any $1 < q < 2$.
\end{theorem}
The notation $\rightharpoonup$ and $\rightarrow$ stands for the weak and strong convergences, respectively. {Let us stress that we work here with the Bessel potential space $H^{s,p}$ which is different from the Besov-space $W^{s,p}$, applications to which we discuss in \Cref{s:remarks}}. {Also, we do not discuss the case $q=1$, because of competing notions on how to properly define it}. A proof to \Cref{th:convergence} will be provided in \Cref{s:th:conv}. The proof presented here contains the one from Brezis in \cite{B81}. {As in Brezis' result the key idea is still to use the fact that non-negative distributions are measures, \Cref{measurefact}}. In addition, as a consequence of Theorem \ref{th:convergence}, we prove the second main result of the paper, a fractional version of the Boccardo/Murat result
\begin{theorem}\label{th:Moscoconvergence}
Let $\Omega \subset \R^N$ be any open bounded set with Lipschitz boundary and let $\psi_n \rightharpoonup \psi$  in $H^{s,q}_{00}(\Omega)$ for {a} $q>2$. Then 
\begin{equation*}
 K_s(\psi_n) := \left \{v \in H^{s,2}_{00}(\Omega):  v \geq \psi_n \quad \text{a.e.} \right \},
\end{equation*}
converges in the sense of Mosco to
\begin{equation*}
 K_s(\psi) := \left \{v \in H^{s,2}_{00}(\Omega):  v \geq \psi \quad \text{a.e.} \right \}.
\end{equation*}
\end{theorem}
We delay the introduction of the concept of Mosco convergence until Section \ref{sec:Mosco}.

The paper is organized as follows: In \Cref{s:prelims}, we introduce the relevant notation and state some well-known properties of the Riesz potential and its relationship with the fractional Laplacian. We also state a critical density result. Our main work starts from \Cref{s:sob}, where we provide a crucial Sobolev embedding type result in \Cref{pr:sob2}, followed by a H\"older type inequality in \Cref{lem:Hineq}. \Cref{s:char} is devoted the characterization of the dual spaces of fractional order Sobolev spaces. The proof of our main result is given \Cref{s:th:conv}, where we divide the task in two parts: first we focus on the interior and then extend the result to the boundary. Finally, we conclude the paper by considering an application of our main result {(\Cref{sec:Mosco}) to convergence of convex sets and quasi-variational inequalities.}

\section{Notation and Preliminaries}\label{s:prelims}
In this section we establish relevant notation and assumptions considered throughout the paper. In particular, we assume that 
\begin{equation}\label{eq:boundedomega}
\Omega \quad \text{ is an open and bounded subset of } \mathbb{R}^N,	
\end{equation}
with Lipschitz boundary.

For two open sets $\Omega_1, \Omega_2$ we say that $\Omega_1$ is compactly contained in $\Omega_2$, in symbols $\Omega_1 \subsubset \Omega_2$, if $\overline{\Omega}_1$ is compact and $\overline{\Omega}_1 \subset \Omega_2$.

{For the ease of presentation we follow the $\aleq$, $\ageq$-notation.} Namely, for two nonnegative {expressions} $A,B \geq 0$ we write $A \aleq B$ if there exists a \emph{non-essential} constant $C > 0$ such that $A \leq C\, B$. The value of the constant $C$ might change at each occurrence. We write $A \aeq B$ if $A \aleq B$ and $B \aleq A$. 

We recall the definition of the $s$-Laplacian $\laps{s}$ in $\R^N$. Denoting $\mathcal{F}$ the Fourier transform, $\laps{s} f$ is defined for functions $f \in C_c^\infty(\R^N)$ via
\[
 \mathcal{F}(\laps{s} f)(\xi) = c |\xi|^s \mathcal{F} f(\xi), \quad \xi \in \R^N.
\]
The constant $c$ can be explicitly computed depending on $s$ and $N$. It is determined so that $\laps{2} = -\lap$ where $\lap$ is the usual Laplacian. The constant $c$ plays however no role in our arguments, cf. \cite{DNPV12}.

The inverse of the fractional Laplacian is defined as the Riesz potential  
\begin{equation*}
	\lapms{s} := (-\lap)^{-\frac{s}{2}},
\end{equation*}
and for $s = 2$ this is usually called the Newton potential. Further, we observe that
\[
 \mathcal{F}(\lapms{s} f)(\xi) = (c |\xi|^s)^{-1} \mathcal{F} f(\xi), \quad \xi \in \R^N.
\]

Both the fractional Laplacian and Riesz potential have a integral representation. For $s \in (0,N)$,
\[
 \lapms{s} f(x) = c\int_{\R^N} |x-y|^{s-N}\, f(y)\, dy. 
\]
Moreover, for $s \in (0,1)$, 
\[
 \laps{s} f(x) = c\int_{\R^N} |x-y|^{-s-N}\, (f(y)-f(x))\, dy.
\]
The latter representation still holds when $s \in (1,2)$ but in a principal value (p.v.) sense. Another representation for $s \in (0,2)$ is given by
\[
 \laps{s} f(x) = \frac{1}{2} c\int_{\R^N} |h|^{-s-N}\, (f(x+h)+f(x-h)-2f(x))\, dh .
\]
See \cite{DNPV12} for an overview, \cite{S02} for more detailed and extended results, and \cite[Section 6.1.1]{G14m} for the harmonic analysis aspects of these operators.

We collect a few basic properties that we will use throughout the paper.
\begin{lemma}\label{la:basicprops}
Let $f,g \in C_c^\infty(\R^N)$. Recalling that for $s > 0$ we have $(-\lap)^{-\frac{s}{2}} = \lapms{s}$ and defining $\laps{0} := id$ we have
\begin{enumerate}[$(a)$]
\item For any $s \in (-N,N)$ an integration-by-parts formula for Riesz potential and fractional Laplacian holds, namely
\[
\int_{\R^N} \laps{s} f\, g =  \int_{\R^N} f\, \laps{s} g \, .
\]
\item $\laps{s} \laps{t} f = \laps{s+t} f$ for all $s,t \in (-N,N)$ with $s+t \in (-N,N)$.
\item For $x \neq 0$ we have $\laps{s} |x|^{t-N} = c |x|^{t-s-N}$ for $s,t \in (0,N)$ with $s \neq t$.
\item For $f,g \in C_c^\infty(\R^N)$ and $s \in (0,2)$ we have the product rule
\begin{align*}
 \laps{s} (fg)(x) &= (\laps{s} f)(x)\, g(x)+f(x)\, (\laps{s} g(x)) \\
 		&\quad +c\int_{\R^N} \frac{(f(x)-f(y))(g(x)-g(y))}{|x-y|^{N+s}}\, dy \, .
\end{align*}

\end{enumerate}
\end{lemma}
\begin{proof}
This follows, e.g., from the definition of the operators via the Fourier transform or the potential represenation.
\begin{enumerate}[$(a)$]
 \item By Plancherel's theorem
 \[
 \begin{split}
\int_{\R^N} {\laps{s} f}\, g =&  \int_{\R^N} \mathcal{F}(\laps{s} f)(\xi)\, \overline{\mathcal{F}(g)(\xi)}\, d\xi\\
=&  c\int_{\R^N} |\xi|^s \mathcal{F}(f)(\xi)\, \overline{\mathcal{F}(g)(\xi)}\, d\xi\\
=&  c\int_{\R^N} \mathcal{F}(f)(\xi)\, \overline{|\xi|^s \, \mathcal{F}(g)(\xi)}\, d\xi\\
=&  c\int_{\R^N} \mathcal{F}(f)(\xi)\, \overline{ \, \mathcal{F}(\laps{s} g)(\xi)}\, d\xi\\
=&  \int_{\R^N} f\, \laps{s} g \, .
\end{split}
 \]
\item We have \[
\begin{split}
    \mathcal{F}(\laps{s} \laps{t} f)(\xi) =& c|\xi|^s \mathcal{F}(\laps{t} f)(\xi) = 
               c|\xi|^s |\xi|^t \mathcal{F}(f)(\xi) \\
               =&c|\xi|^{s+t} \mathcal{F}(f)(\xi) =c\mathcal{F}(\laps{s+t}f)(\xi).
               \end{split}
              \]
Inverting the Fourier transform on both sides we obtain the claim.
\item One can check that the Fourier-transform of a $\sigma$-homogeneous function ($\sigma < N$, $\sigma \neq 0$) is $-N-\sigma$-homogeneous. Moreover the Fourier transform preserves radial functions. From this we conclude that $\mathcal{F} (|\cdot|^{t-N}) (\xi) = c|\xi|^{-t}$. Consequently, $\mathcal{F} (\laps{s}  |\cdot|^{t-N})(\xi) = |\xi|^s |\xi|^{-t} = |\xi|^{s-t}$. If $s \neq t$ then we can invert the Fourier transform and have that $\laps{s} |\cdot|^{t-N} = c |\cdot|^{t-s-N}$.
\item This follows from the potential representation of $\laps{s}$, observing that 
\[
\begin{split}
 f(x)g(x)-f(y)g(y) &= f(x)\, (g(x)-g(y))+g(y)\, (f(x)-f(y))\\
 &= f(x)\, (g(x)-g(y))+g(x) (f(x)-f(y)) \\
 &\quad- (g(x)-g(y))\, (f(x)-f(y))
 \end{split}
\]
\end{enumerate}
The proof of \Cref{la:basicprops} is complete.
\end{proof}

Next we recall the fractional Sobolev space $H^{s,p}(\R^N)$, also referred to as Besov space (it corresponds to the Besov space $B^s_{p,p}(\R^N)$, \cite{RS96}). 

For $s \in (0,\infty)$, $p \in (1,\infty)$ denote
\[
 [f]_{H^{s,p}(\R^N)} = \|\laps{s} f\|_{L^p(\R^N)},
\]
\[
 \|f\|_{H^{s,p}(\R^N)} := \|f\|_{L^p(\R^N)} + [f]_{H^{s,p}(\R^N)}.
\]
Then the space $H^{s,p}(\R^N)$ is defined as
\[
 H^{s,p}(\R^N) = \left \{f \in L^p(\R^N):\quad \|f\|_{H^{s,p}(\R^N)} < \infty \right \}.
\]
We also need to define the fractional Sobolev space on open sets $\Omega \subset \R^N$ with zero boundary data on $\partial \Omega$. Since $\laps{s}$ acts on functions defined in all of $\R^N$ one needs to assume that the functions vanish on the complement of $\Omega$. More precisely, for $s \in (0,N)$ we define
\begin{equation}\label{eq:Hsp00}
 H_{00}^{s,p}(\Omega) := \left \{ f \in L^p(\R^N), \quad  \laps{s} f \in L^p(\R^N), \quad f \equiv 0 \text{ in $\Omega^c$} \right \}
\end{equation}
endowed with the norm ${\|f\|_{H^{s,p}_{00}(\Omega)} := \|f\|_{H^{s,p}(\R^N)} =} \|f\|_{L^p(\R^N)} + \|\laps{s} f\|_{L^p(\R^N)}$.
It is useful to observe that this definition makes sense for \emph{any} open set $\Omega \subset \R^N$, there is no requirement on the regularity of its boundary $\partial \Omega$.

For the topological dual $H^{s,p}_{00}(\Omega)^{\ast}$, throughout the paper, we use the following notation 
\begin{equation*}
	H^{-s,p'}(\Omega) = H^{s,p}_{00}(\Omega)^{\ast},
\end{equation*}
where $p'$ is the H\"{o}lder conjugate of $p$. Throughout the paper, and abusing notation, we denote the duality pairing between elements $h\in H^{-s,p'}(\Omega)$ and $u\in H^{s,p}_{00}(\Omega)$ by {$h[u]$}.

{For a given Banach space $X$, the strong and weak convergences of a sequence $\{z_n\}$ in $X$ to some $z\in X$ are denoted ``$z_n\to z$'' and ``$z_n\rightharpoonup z$'', respectively.}

For $s=1$ and $p \in (1,\infty)$ the space $H^{1,p}(\R^N)$ coincides with the usual Sobolev space. If $\Omega \subset \R^N$ with Lipschitz boundary $\partial \Omega$ then $H^{1,p}_{00}(\Omega)$ coincides with the usual Sobolev space of functions with trace zero on $\partial \Omega$. Clearly, $C_c^\infty(\Omega) \subset H^{s,p}_{00}(\Omega)$ for any $s \in (0,1]$, $p \in (1,\infty)$. 
\begin{theorem}\label{th:density}
Let $s \in (0,1]$, $p \in (1,\infty)$ and $\Omega \subset \R^N$ be any open set. Then $C_c^\infty(\Omega)$ is dense in $H^{s,p}_{00}(\Omega)$.
\end{theorem}
While this statement is relatively easy to prove for $\Omega$ under the assumption that $\partial \Omega$ is sufficiently regular, it is more involved for generic open set $\Omega$, even if $s=1$. \Cref{th:density} is proven in \cite[Theorem 10.1.1]{AH96}; they attribute the proof to \cite{N92}.

\section{Sobolev inequalities}\label{s:sob}
In this section we state embedding and compactness results which are crucial to the remainder of the paper. 
We shall exploit a limiting version of Sobolev embedding stated in \Cref{pr:sob2}.

We begin by stating a Sobolev inequality.
\begin{lemma}[Sobolev inequality]\label{la:sobolev}
Assume that $\laps{s} f \in L^p(\R^N)$ for $s \in (0,N)$ and $p \in (1,\infty)$. Then for any $t \in [0,s]$, if $s-t-\frac{N}{p} \in (0,1)$ then
\[
 \|\laps{t} f\|_{L^{\frac{Np}{N-(s-t)p}}(\R^N)} \leq C(t,p)\, \|\laps{s} f\|_{L^p(\R^N)}.
\]
Moreover for $g \in L^p(\R^N)$, $s \in (0,N)$ with $\frac{Np}{N-sp} \in (1,\infty)$,
\[
 \|\lapms{s} g\|_{L^{\frac{Np}{N-sp}}(\R^N)} \leq C(s,p)\, \|g\|_{L^p(\R^N)}.
\]
\end{lemma}
\Cref{la:sobolev} follows e.g. from \cite[Theorem 6.1.3.]{G14m}.

\begin{proposition}[Rellich]\label{pr:rellich}
The following results hold:
\begin{enumerate}
 \item Assume that for some $1 < q \leq p < \infty$ and some $t > 0$
\[
 \sup_{k} \|\laps{t} f_k\|_{L^q(\Omega)} + \|f_k\|_{L^p(\R^N)} < \infty
\]
then $f_k$ converges (up to subsequence) strongly to some $f$ in $L^q(\Omega')$ for any $\Omega' \subset \subset \Omega$.
\item Assume that $f_k \in L^p(\R^N)$ with
\[
 \sup_{k} \|f_k\|_{L^p(\R^N)} < \infty.
\]
Set $g_k := \lapms{s} f_k$ where $\lapms{s}$ is the Riesz potential. Then  
up to a subsequence, $g_k$ is a Cauchy sequence in $L^p(\Omega)$.
\end{enumerate}
\end{proposition}
\Cref{pr:rellich} can be proven either by a direct adaptation of the proof of the Rellich-Kondrachov theorem for classical Sobolev spaces, \cite[Section 5.7]{E10}. Or it can be derived from the abstract theory of Besov and Triebel spaces, see \cite[Section 2.4.4]{RS96}.

\begin{lemma}\label{la:poincare}
Let $\Omega \subsubset \R^N$. Then there exists a constant $C > 0$ depending on $s$, $N$, $p$ and the diameter of $\Omega$ such that for any $f \in H^{s,p}_{00}(\Omega)$,
\[
 \|f\|_{L^p(\R^N)} \leq C\, \|\laps{s} f\|_{L^p(\R^N)}.
\]
\end{lemma}
\Cref{la:poincare} follows from \Cref{pr:rellich} via the usual blow-up proof for Poincar\'e inequality. Due to \Cref{la:poincare}, for $f \in H_{00}^{s,p}(\Omega)$, we have the following norm-equivalence $\|f\|_{H_{00}^{s,p}(\Omega)} \aeq \|\laps{s} f\|_{L^p(\R^N)}$. Thus 
	\begin{equation}\label{eq:semi_norm}
		 \| f \|_{H_{00}^{s,p}(\Omega)} := \|\laps{s} f\|_{L^p(\R^N)} \, ,
	\end{equation}
	defines a norm on $H_{00}^{s,p}(\Omega)$.

We will make crucial use of the following adaptation of the Sobolev embedding.
\begin{proposition}\label{pr:sob2}
Let $\Omega_1 \subset \subset \Omega_2 \subset \R^N$ be two open sets, $p \in (1,\infty)$, $s \in (0,1)$ and $t \in (s,2s)$.
Then for any $q \in [1,\frac{N}{N-(2s-t)})$, $q \leq p$ we have
\[
 \|\laps{t} f\|_{L^{q} (\Omega_1)} \leq C(\Omega_1,\Omega_2,p,q,N) \brac{\|\laps{2s} f\|_{L^1(\Omega_2)} + \|\laps{s} f\|_{L^p(\R^N)}}
\]
whenever $f \in L^p(\R^N)$ is a function for which the right-hand-side is finite.
\end{proposition}
\begin{proof}
We may assume w.l.o.g. that $\Omega_2 \subset \subset \R^N$.

Recall that in view of \Cref{la:basicprops},
\[
 \laps{t} f = \lapms{2s-t} \laps{2s} f.
\]
Take $\eta \in C_c^\infty(\Omega_2)$ with $\eta \equiv 1$ in $\Omega_1$ a suitable cutoff function. Then,
\begin{equation}\label{sobl2:1}
 \laps{t} f = \lapms{2s-t} \brac{\eta \laps{2s} f} + \lapms{2s-t} \brac{(1-\eta) \laps{2s} f} . 
\end{equation}
Now observe that for $x \in \Omega_1$ we can write
\[
 \lapms{2s-t} (\eta \laps{2s} f)(x) = k\, \ast (\eta \laps{2s} f)(x)
\]
where for some $\Lambda > 0$ depending on diameter of $\Omega_1 \subset \supp \eta$,
\[
 k(z) = \chi_{\left\{ z\ : \ |z| \le c \Lambda \right\}} |z|^{2s-t-N} ,
\]
where $c > 0$ is a generic constant.
One can check that $k \in L^{q}(\R^N)$ for any $1 \leq q < \frac{N}{N-(2s-t)}$. By Young's inequality for convolutions,
\begin{equation}\label{sob2:56}
 \|k\, \ast \brac{\eta (\laps{2s} f)}\|_{L^{q}(\R^N)} \aleq C(k)\|\eta \laps{2s} f\|_{L^1(\R^N)}  \aleq C(k,\eta)\, \|\laps{2s} f\|_{L^1(\Omega_2)}.
\end{equation}
That is,
\[
 \|\lapms{2s-t} \brac{\eta \laps{2s} f}\|_{L^{q}(\Omega_1)} \aleq  C(k,\eta)\, \|\laps{2s} f\|_{L^1(\Omega_2)}.
\]
This treats the first term in \eqref{sobl2:1}.

It remains to treat the second term in \eqref{sobl2:1}. For $x \in \Omega_1$ we have $1-\eta(x) = 0$. If we define 
\[
 \kappa(x,y) := |x-y|^{2s-t-N} (1-\eta(y))
\]
then we observe that for every fixed $x \in \Omega_1$, $\kappa(x,\cdot)$ is a smooth, bounded function. By \Cref{la:basicprops},
\[
  \lapms{2s-t} \brac{(1-\eta) \laps{2s} f}(x) = \int_{\R^N} (\laps{s}_y \kappa(x,\cdot))(y)\, \laps{s} f(y)\, dy .
\]
From \Cref{la:basicprops} we have $\laps{s} |\cdot|^{2s-t-N} = c\, |\cdot|^{-(t-s)-N}$. Since $t \in (s,2s)$ and $s \in (0,1)$ we then have from the product rule in \Cref{la:basicprops},
\[
 (\laps{s} \kappa(x,\cdot))(y) =: \sum_{i=1}^3 c_i \Gamma_i(x,y) 
\]
where
\[
 \Gamma_1(x,y) := |x-y|^{-(t-s)-N} (1-\eta(y)) 
\]
\[
 \Gamma_2(x,y) := |x-y|^{2s-t-N} \laps{s} \eta(y)
\]
and
\[
 \Gamma_3(x,y) := \int_{\R^N} \frac{(|y-x|^{2s-t-N}-|z-x|^{2s-t-N})(\eta(y)-\eta(z))}{|y-z|^{N+s}}\, dz.
\]
\underline{We first treat $\Gamma_1$}: We have $\dist(\Omega_1,\supp(1-\eta)) > 0$, so we have
\[
 \sup_{x \in \Omega_1} \Gamma_1(x,y) \aleq (1+|y|)^{-(t-s)-N}.
\]
Since $t-s > 0$ the right-hand side is integrable to any power, in particular
\[
 \sup_{x \in \Omega_1} \|\Gamma_1(x,\cdot)\|_{L^{p'}(\R^N)} < \infty,
\]
which implies that
\[
 \sup_{x \in \Omega_1} \int_{\R^N} \Gamma_1(x,y)\, \laps{s} f(y) dy \aleq \|\laps{s} f\|_{L^p(\R^N)}.
\]
Since $\Omega_1$ is bounded, we find
\[
 \brac{\int_{\Omega_1} \left |\int_{\R^N} \Gamma_1(x,y)\, \laps{s} f(y) dy \right |^q}^{\frac{1}{q}} \aleq \|\laps{s} f\|_{L^p(\R^N)}.
\]
\underline{Now we treat $\Gamma_2$}. Since $\eta \in C_c^\infty(\Omega_2)$, for any $x \in \Omega_1$,
\[
 \laps{s} \eta(y) \aleq \frac{1}{1+|y|^{N+2s-t}}.
\]
Moreover for $x \in \Omega_1$ and $y \not \in \Omega_2$ we have $|x-y| \aeq 1+|y|$. Consequently, 
for any $x \in \Omega_1$,
\[
\begin{split}
 |x-y|^{2s-t-N}|\laps{s} \eta(y)| =& \chi_{\Omega_2}(y) |x-y|^{2s-t-N}|\laps{s} \eta(y)|  \\
 &\quad+ \chi_{\R^N \backslash \Omega_2} |x-y|^{2s-t-N}|\laps{s} \eta(y)| \\
 \aleq& \chi_{\Omega_2}(y)\, |x-y|^{2s-t-N} + \brac{1+|y|}^{-2N}.
 \end{split}
\]
Observe again that the second term can be integrated to any nonnegative power, so that we have for any $x \in \Omega_1$
\[
 \int_{\R^N} \Gamma_2(x,y) \laps{s} f(y) \aleq \lapms{2s-t} \brac{\chi_{\Omega_2} |\laps{s} f|}(x) + \|\laps{s}f\|_{L^p(\R^N)}. 
\]
Using again Young's inequality as in \eqref{sob2:56}, we then find
\[
\begin{split}
 &\brac{\int_{\Omega_1} \left |\int_{\R^N} \Gamma_2(x,y)\, \laps{s} f(y) dy \right |^q}^{\frac{1}{q}} \\
 \aleq& \|\lapms{2s-t} \brac{\chi_{\Omega_2} |\laps{s} f|}\|_{L^q(\Omega_1)}+ \|\laps{s}f\|_{L^p(\R^N)}\\
 \aleq& \|\chi_{\Omega_2}\, |\laps{s} f|\|_{L^q(\R^N)}  +\|\laps{s}f\|_{L^p(\R^N)}\\
\aleq& \|\laps{s}f\|_{L^p(\R^N)}\\ 
\end{split}
 \]
It remains to \underline{treat $\Gamma_3$},
\[
 \Gamma_3(x,y) = \int_{\R^N} \frac{(|y-x|^{2s-t-N}-|z-x|^{2s-t-N})(\eta(y)-\eta(z))}{|y-z|^{N+s}}\, dz.
\]
Firstly, if $y \not \in \Omega_2$ then $\eta(y)-\eta(z) =0$ unless $z \in \supp \eta$. Moreover, $y \not \in \Omega_2$ and $z \in \supp \eta$ means that $|z-x| \aleq |y-x|$ and $(1+|y|) \aeq |y-z|$. Thus, since $2s - t > 0$
\[
\sup_{x \in \Omega_1}\, \chi_{\R^N \backslash \Omega_2}(y)| \Gamma_3(x,y)| \aleq \frac{1}{1+|y|^{N+s}}\, \int_{\supp \eta} |z-x|^{2s-t-N}\, dz \aleq \frac{1}{1+|y|^{N+s}}.
\]
That is,
\begin{equation}\label{eq:sob22}
\begin{split}
 &\sup_{x \in \Omega_1} |\int_{\R^N \backslash \Omega_2} \Gamma_3(x,y) \laps{s} f(y)\, dy|\\
 \aleq &\int_{\R^N} |f(y)| \frac{1}{1+|y|^{N+s}} \aleq \|\laps{s} f\|_{L^p(\R^N)}\, \|\frac{1}{1+|\cdot|^{N+s}}\|_{L^{p'}(\R^N)} \aleq \|\laps{s} f\|_{L^p(\R^N)}.
 \end{split}
\end{equation}
As for $\chi_{\Omega_2}(y)| \Gamma_3(x,y)|$, we estimate
\begin{equation}\label{eq:sob23}
\begin{split}
 \chi_{\Omega_2}(y)| \Gamma_3(x,y)| \aleq& \chi_{\Omega_2}(y)\int_{\dist(z,\Omega_2) > 1}  \frac{\left ||y-x|^{2s-t-N}-|z-x|^{2s-t-N}\right |\, \left |\eta(y)-\eta(z)\right|}{|y-z|^{N+s}}\, dz \\
 &+  \chi_{\Omega_2}(y)\int_{\dist(z,\Omega_2) \leq 1}\frac{\left ||y-x|^{2s-t-N}-|z-x|^{2s-t-N}\right |\, \left |\eta(y)-\eta(z)\right|}{|y-z|^{N+s}}\, dz.
 \end{split}
\end{equation}
Observe that if $y \in \Omega_2$, $x \in \Omega_1$ and $\dist(z,\Omega_2) > 1$ then $|y-z| \aeq 1+|z|$, and $|y-x| \aleq |z-x|$. Thus,
\[
\begin{split}
&\chi_{\Omega_2}(y) \int_{\dist(z,\Omega_2) > 1}  \frac{\left ||y-x|^{2s-t-N}-|z-x|^{2s-t-N}\right |\, \left |\eta(y)-\eta(z)\right|}{|y-z|^{N+s}}\, dz \\
\aleq&\chi_{\Omega_2}(y)|y-x|^{2s-t-N} \frac{1}{1+|z|^{N+s}}\, dz \aeq \chi_{\Omega_2}(y)|y-x|^{2s-t-N}.
 \end{split}
\]
On the other hand, if $\dist(z,\Omega_2) \leq 1$ then either $|y-x| \leq 10|z-x|$ or $|y-x| \aeq |z-y|$, and thus (using $s < 1$)
\[
\begin{split}
 &\chi_{\Omega_2}(y)\int_{\dist(z,\Omega_2) \leq 1}  \frac{\left ||y-x|^{2s-t-N}-|z-x|^{2s-t-N}\right |\, \left |\eta(y)-\eta(z)\right|}{|y-z|^{N+s}}\, dz\\
  \aleq &\|\nabla \eta\|_{L^\infty}\, \chi_{\Omega_2}(y)\int_{\dist(z,\Omega_2) \leq 1}  \, \frac{\left ||y-x|^{2s-t-N}-|z-x|^{2s-t-N}\right |}{|y-z|^{N+s-1}}\, dz\\
 \aleq &|y-x|^{2s-t-N}\, \chi_{\Omega_2}(y)\int_{\dist(z,\Omega_2) \leq 1}  \frac{1}{|y-z|^{N+s-1}}\, dz\\
 &+\chi_{\Omega_2}(y) \frac{1}{|x-y|^{N+s-1}}\, \int_{|z-x| \aleq |x-y|} |z-x|^{2s-t-N}\, dz\\
  \aleq &|y-x|^{2s-t-N}\, \chi_{\Omega_2}(y)+\chi_{\Omega_2}(y) \frac{|x-y|^{2s-t}}{|x-y|^{N+s-1}}\\
  \aleq&\brac{|y-x|^{2s-t-N}+|y-x|^{s+1-t-N}} \chi_{\Omega_2}(y) \, .
 \end{split}
\]
In second to the last step we used that $x \in \Omega_1$, and $\Omega_1,\Omega_2$ are both bounded sets.

All in all, from \eqref{eq:sob23} we arrive for any $x \in \Omega_1$,
\[
\begin{split}
 &|\int_{\Omega_2} \Gamma_3(x,y) \laps{s} f(y)\, dy|\\
 \aleq &\lapms{2s-t} (\chi_{\Omega_2} |\laps{s} f|)(x)+\lapms{s+1-t} (\chi_{\Omega_2} |\laps{s} f|)(x).
 \end{split}
\]
Arguing as in \eqref{sob2:56}, for $\sigma = 2s-t$ or $\sigma = s+1-t$ (in particular $\sigma \in (0,1)$) we have
\[
 \|\lapms{s} (\chi_{\Omega_2} |\laps{s} f|)\|_{L^q(\R^N)} \aleq \|\chi_{\Omega_2} |\laps{s} f|\|_{L^1(\R^N)} \aleq \|\laps{s} f\|_{L^p(\R^N)}.
\]
Together with estimate \eqref{eq:sob22} we arrive at
\[
 \| x \mapsto \int_{\R^N} \Gamma_3(x,y) \laps{s} f(y) dy\|_{L^q_x(\Omega_1)} \aleq \|\laps{s} f\|_{L^p(\R^N)}.
\]
This finishes the proof.
\end{proof}

\begin{lemma}[A type of H\"older's inequality]
\label{lem:Hineq}
Let $1 < q < p < \infty$ and $\Omega \subset \subset \R^N$, $s \in (0,1)$. Then 
\[
 \|\laps{s} \varphi\|_{L^q(\R^N)} \leq C(\Omega,s,p,q)\, \|\laps{s} \varphi\|_{L^p(\R^N)}
\]
holds for all $\varphi \in C_c^\infty(\Omega)$.
\end{lemma}
\begin{proof}
Fix $\Omega_1 \subsubset \R^N$ such that $\Omega \subsubset \Omega_1$ and take $\chi_1, \chi_2 \in C_c^\infty(\Omega_1)$ with both $\chi_1,\chi_2 \equiv 1$ in a neighborhood of $\Omega$, but so that $\chi_1(x) (1-\chi_2(x)) = 0$. Then, by H\"older's inequality
\[
  \|\chi_2 \laps{s} \varphi\|_{L^q(\R^N)} \aleq \|\laps{s} \varphi\|_{L^p(\R^N)}.
\]
Moreover, by the support of $1-\chi_2$, $\chi_1$ and $\varphi$,
\[
 (1-\chi_2(x)) \laps{s} \varphi(x) = \int_{\R^N} (1-\chi_2(x)) \chi_1(y)\, |x-y|^{-s-N} \varphi(y)\, dy.
\]
Let 
\[
 k(z) := \begin{cases}
            |z|^{-s-N} \quad& \text{if } |z| \geq \frac{1}{2} \dist(\supp (1-\chi_2),\supp (\chi_1))\\
            0\quad \text{otherwise}\\
           \end{cases}
\]
then
\[
 \left |(1-\chi_2(x)) \laps{s} \varphi(x)\right | \aleq  \int_{\R^N} k(z) \varphi(y)\, dy.
\]
Since $k \in L^1(\R^N)$ we have by Young's inequality for convolutions,
\[
 \|(1-\chi_2) \laps{s} \varphi\|_{L^q} \aleq \|\varphi\|_{L^q(\R^N)}.
\]
By H\"older's and then Poincar\`e inequality, \Cref{la:poincare}, recall that $\supp \varphi \subset \Omega$,
\[
 \|\varphi\|_{L^q(\R^N)} \aleq \|\varphi\|_{L^p(\R^N)} \aleq \|\laps{s} \varphi\|_{L^p(\R^N)}.
\]
The proof is complete.
\end{proof}

\section{Characterization of the Dual of \texorpdfstring{$H^{s,p}$}{Hsp}}
\label{s:char}

We will also make use of the classification of the dual space of $H^{s,p}$-spaces. The following proposition is a well known-result, indeed it is a consequence of Hahn-Banach theorem and Helmholtz-Hodge decomposition. {We state it for the classical Sobolev space $H^{1,p}_{0}(\Omega)= \overline{C_c^\infty(\Omega)}^{\|\cdot\|_{H^{1,p}}}$. If $\partial \Omega$ is regular enough (namely, if the trace and extension theorem hold) this classical space coincides with $H^{1,p}_{00}(\Omega)$ as we defined it}.
\begin{proposition}
Let $p \in (1,\infty)$, $p'=\frac{p}{p-1}$. For {each} $h \in H^{1,p}_0(\Omega)^\ast$, there exists $f \in L^{p'}(\Omega)$ and $g \in W^{1,p'}(\Omega)$ such that 
\[
 h[\varphi] = \int_{\Omega} f\cdot \varphi + \int_{\Omega} \nabla g\cdot \nabla \varphi 
\]
with
\[
 \|f\|_{L^{p'}(\Omega)} + \|g\|_{W^{1,p'}(\Omega)} \leq C(\Omega,p) \|h\|_{H^{1,p}_0(\Omega)^\ast} \, .
\]
\end{proposition}
It is known that this can also be done for fractional Sobolev spaces, cf. \cite[Theorem 5.3]{AVW19} and \cite[Proposition 2.2.]{MSY20}. 

\begin{proposition}[Characterization of $H^{s,p}(\R^N)^\ast$]\label{pr:dualcharglobal}
Let $p \in (1,\infty)$, $p'=\frac{p}{p-1}$, and $s \in (0,N)$. {If} $h \in H^{s,p}(\R^N)^\ast$ then there exists $f_1 \in L^{p'}(\R^N)$ and $f_2 \in H^{s,p'}(\R^N)$ such that 
\[
 h[\varphi] = \int f_1 \varphi + \int \laps{s} f_2 \laps{s} \varphi, \qquad \forall \varphi \in C_c^\infty(\R^N)
\]
and
\[
 \|f_1\|_{L^{p'}(\R^N)} + \|f_2\|_{L^{p'}(\R^N)} + \|\laps{s} f_2\|_{L^{p'}(\R^N)} \le C(s,p,N) \|h\|_{H^{s,p}(\R^N)^\ast} \, .
\]
\end{proposition}
\begin{proof}
We consider $H^{s,p}(\R^N)$ as a linear subspace $X \subset L^p(\R^N)\times L^p(\R^N)$ via the identification
\[
 T: f \mapsto (f,\laps{s}f),
\]
with $X := T(H^{s,p}(\R^N))$. Then $h \in H^{s,p}(\R^N)^\ast$ induces a linear bounded functional $h^\ast$ on $X$ via,
\[
 h^\ast (f,\laps{s} f) := h[f],
\]
with
\[
 \|h^\ast \|_{X^\ast} = \|h\|_{H^{s,p}(\R^N)^\ast} \, .
\]
By Hahn-Banach theorem there exists an extension of $h^\ast$ to all of $L^{p}(\R^N) \times L^{p}(\R^N)$. By Riesz representation theorem there must be an $f \in L^{p}(\R^N)$, $g \in L^{p}(\R^N)$ such that 
\[
 h^\ast(\varphi,\psi) = \int_{\R^N} f \varphi + \int_{\R^N} g \psi \quad \forall \varphi,\psi \in L^p(\R^N).
\]
In particular,
\begin{equation}\label{eq:char:234}
  h(\varphi) = h^\ast(\varphi,\laps{s} \varphi) = \int_{\R^N} f \varphi + \int_{\R^N} g \laps{s} \varphi \quad \forall \varphi \in C_c^\infty(\R^N).
\end{equation}
Now we define the operator $B_s := (I+\laps{s})^{-1}$, that is for $u \in C_c^\infty(\R^N)$, 
\[
 \mathcal{F}(B_s(u))(\xi) := \frac{1}{1+c|\xi|^s} \mathcal{F}u(\xi).
\]
This operator is a slight variation of the Bessel potential operator. More precisely, it is a multiplier operator with bounded symbol $\frac{1}{1+c|\xi|^s}$ that satisfies Mihlin's and H\"ormander's condition and thus $B_s: L^q(\R^N) \to L^q(\R^N)$ is a linear bounded operator for all $q \in (1,\infty)$, \cite[Theorem 5.2.7.]{G08c}. 
Moreover, $\laps{s} B_s$ is a bounded linear operator from $L^q(\R^N) \to L^q(\R^N)$. Indeed, this follows from the fact that 
\[
 \mathcal{F}(\laps{s} B_s(u))(\xi) = \frac{c|\xi|^s}{1+c|\xi|^s} \mathcal{F}u(\xi).
\]
Again, the symbol $\frac{c|\xi|^s}{1+c|\xi|^s}$ satisfies Mihlin's and H\"ormander's condition and we can apply H\"ormander's theorem, \cite[Theorem 5.2.7.]{G08c}, to conclude the boundedness of $\laps{s} B_s$. From these two results we obtain 
\begin{equation}\label{eq:hBsest}
 \|\laps{s} B_s u\|_{L^q(\R^N)} +\|B_s u\|_{L^q(\R^N)} \aleq \|u\|_{L^q(\R^N)} \quad \forall u \in L^q(\R^N), \, \forall q \in (1,\infty).
\end{equation}
Moreover, 
\begin{equation}\label{eq:hBsdec}
 u = \laps{s} B_s u + B_s u \qquad \forall u \in L^q(\R^N), \quad \forall q \in (1,\infty);
\end{equation}
indeed applying the Fourier transform, \eqref{eq:hBsdec} is equivalent to 
\[
 c|\xi|^2 \frac{1}{1+c|\xi|^s} \mathcal{F}u(\xi) + \frac{1}{1+c|\xi|^s} \mathcal{F}u(\xi) = \mathcal{F} u(\xi).
\]
Applying \eqref{eq:hBsdec} to $g$ we reformulate \eqref{eq:char:234} into
\[
\begin{split}
  h(\varphi) =& \int_{\R^N} f \varphi + \int_{\R^N} \brac{\laps{s} B_s g + B_s g} \laps{s} \varphi\\
   =&\int_{\R^N} (f  +\laps{s} B_s g)\, \varphi + \int_{\R^N} \laps{s} B_s g \laps{s} \varphi\\
\end{split}
\]
and in view of \eqref{eq:hBsest} we have 
\[
 \|f+\laps{s} B_s g\|_{L^{p'}(\R^N)} \aleq  \|h\|_{H^{s,p}(\R^N)^\ast} \, ,
\]
and
\[
 \|B_s g\|_{H^{s,p'}(\R^N)} \aleq  \|h\|_{H^{s,p}(\R^N)^\ast}.
\]
So if we set $f_1 := f+\laps{s} B_s g$ and $f_2 := B_s g$ the claim of \Cref{pr:dualcharglobal} is proven.
\end{proof}

For $p=2$ we also get a local version of \Cref{pr:dualcharglobal}. We restrict our attention to $p = 2$, since for $p \neq 2$ the estimate \eqref{eq:id:est} requires the (to our knowledge unknown) $L^p$-boundary Calder\'{o}n-Zygmund regularity theory for nonlocal differential equations -- for the interior regularity theory cf. \cite{MSY20}.

\begin{proposition}[Identification of $H^{s,2}_{00}(\Omega)^\ast$]\label{pr:id}
Let $\Omega$ be any open bounded set, $s \in (0,1)$,
\[
 h \in H^{s,2}_{00}(\Omega)^\ast
\]
then there exists $u \in H^{s,2}_{00}(\Omega)$ such that for any $\varphi \in C_c^\infty(\Omega)$,
\begin{equation}\label{eq:hs200ast:1}
 h[\varphi] = \int_{\R^N} \laps{s} u\, \laps{s} \varphi . 
\end{equation}
Moreover,
\begin{equation}\label{eq:id:est}
 \|\laps{s} u\|_{L^2(\R^N)} = \|h\|_{H^{s,2}_{00}(\Omega)^\ast}.
\end{equation}
\end{proposition}

\begin{proof}
The proof is a direct consequence of the fact that $\|\laps{s} \cdot\|_{L^2(\R^N)}$ defines {a  Hilbert-space norm} on $H^{s,2}_{00}(\Omega)$ (see \eqref{eq:semi_norm}) and the Riesz Representation Theorem. 
\end{proof}

\section{Proof of Convergence, Theorem~\ref{th:convergence}}\label{s:th:conv}

\Cref{th:convergence} is a essentially equivalent of the following statement, which is proven within this section.  Recall that $\Omega\subset \mathbb{R}^N$ is an arbitrary bounded and open set with Lipschitz boundary. In the following we continue to use the following notation, $H^{-s,q}(\Omega) = H^{s,q'}_{00}(\Omega)^{\ast}$.
\begin{theorem}\label{th:convergence2}
Suppose that $h_n,h \in H^{-s,2}(\Omega) = H_{00}^{s,2}(\Omega)^\ast$, for $s \in (0,1]$ and such that 
\begin{itemize}
 \item $h_n, h \geq 0$ in the sense of distributions in $\Omega$, that is
 \[
  h_n[\varphi],\ h[\varphi] \geq 0 \qquad \forall \varphi \in C_c^\infty(\Omega), \, \varphi \geq 0.
 \]
 \item $h_n$ weakly converges to $h$ in $H_{00}^{s,2}(\Omega)^\ast$, that is
\[
  h_n[\varphi] \xrightarrow{n\to \infty} h[\varphi] \qquad \forall \varphi \in H^{s,2}_{00}(\Omega).
 \]
\end{itemize}
Then {$h_n \xrightarrow{n \to \infty} h$} strongly in $H^{-s,q}(\Omega) = H^{s,q'}_{00}(\Omega)^{\ast}$ for any $1 < q < 2$, i.e.
\begin{equation}\label{eq:conv2:conclusion}
\sup_{\varphi \in C_c^\infty(\Omega),\, \|\varphi\|_{H^{s,q'}(\R^N)}\leq 1} \left |h_n[\varphi]-h[\varphi]\right | \xrightarrow{n \to \infty} 0.
\end{equation}
\end{theorem}

\Cref{th:convergence2} is obtained as a consequence of \Cref{th:convergenceint} which shows the strong convergence of the functionals when localized  away from the boundary $\partial \Omega$. We further prove that the latter result holds not only for $H_{00}^{s,2}(\Omega)^\ast$ but also for $H_{00}^{s,p}(\Omega)^\ast$ with $p\in (1,+\infty)$. Finally, this  together with  \Cref{th:uptobd} allows to extend the result and obtain \Cref{th:convergence2}.

{
\begin{example}
Let us remark that the statement of \Cref{th:convergence2} is false in general without the assumption $h_n, h \geq 0$. E.g., let $\Omega := B(0,1)$, and take any $f_n \in L^2(\mathbb{R}^N)$, $\supp f_n \subset B(0,1)$ with $f_n$ converging to some $f$ weakly in $L^2(\R^N)$, but no subsequence of $f_n$ converges strongly to $f$ in any $L^q(\R^N)$, $1 \leq q \leq 2$. Take for instance, the usual example for weak but not strong convergence \[f_n(x) := \eta(x) \cos(nx_1)\cos(nx_2)\ldots \cos(nx_N),\]
where $\eta \in C_c^\infty(\Omega)$ is a standard radial bump function.

Let $s \in (0,1]$ and set $h_n := \laps{s} f_n$. For any $q \in (1,2)$, 
\[
 \liminf_{n \to \infty} \sup_{\varphi \in C_c^\infty(\Omega),\, \|\varphi\|_{H^{s,q'}(\R^N)}\leq 1}  |h_n[\varphi]-h[\varphi]| \geq \liminf_{n \to \infty} \|f_n-f\|_{L^q(\R^N)} > 0.
\]
That is, the conclusion \eqref{eq:conv2:conclusion} of \Cref{th:convergence2} fails -- even though $h_n$ satisfies all assumptions of \Cref{th:convergence2} besides $h_n[\varphi] \geq 0$. 
\end{example}
}

\subsection{Local strong convergence}
We prove first the interior result which holds for any $p \in (1,\infty)$.
\begin{theorem}\label{th:convergenceint}
Suppose that $h_n,h \in H^{-s,p}(\Omega) = H_{00}^{s,p'}(\Omega)^\ast$, $s \in (0,1]$, $p \in (1,\infty)$ such that 
\begin{itemize}
 \item $h_n, h \geq 0$ in the sense of distributions in $\Omega$. 
  \vspace{.2cm} 
 \item $h_n$ weakly converges to $h$ in $H_{00}^{s,p'}(\Omega)^\ast$. 
\end{itemize}

For an arbitrary $\eta \in C_c^\infty(\Omega)$, consider
\[
 \tilde{h}_{n,\eta}[\varphi] := h_n[\eta \varphi],
\]
and
\[
 \tilde{h}_{\eta}[\varphi] := h[\eta \varphi].
\]
Then, for any open and bounded $\Omega' \subset \R^N$ (not necessarily contained in $\Omega$) we have strong convergence of $\tilde{h}_{n,\eta}$ to $\tilde{h}_{\eta}$ in $H^{-s,q}(\Omega') = H^{s,q'}_{00}(\Omega')^*$, for any $q \in (1,p)$; that is
\begin{equation}\label{eq:interiorconv}
 \|\tilde{h}_{n,\eta}-\tilde{h}_{\eta}\|_{H^{s,q'}_{00}(\Omega')^\ast} \xrightarrow{n \to \infty} 0.
\end{equation}

\end{theorem}

\begin{proof}[Proof of \Cref{th:convergenceint}] 
First note that the result need only to be proven for a subsequence $\tilde{h}_{n,\eta}$, and then the convergence for the complete sequence is given by the uniqueness of the limit $h$. 

Let $\varphi\in H_{00}^{s,p'}(\Omega)$ be arbitrary, then we observe 
\begin{equation}\label{eq:prvintconv1}
 \|\laps{s} \brac{\eta \varphi} \|_{L^{p'}(\R^N)} \aleq C(\eta)\, \brac{\|\varphi\|_{L^{p'}(\R^N)} + \|\laps{s} \varphi\|_{L^{p'}(\R^N)}}.
\end{equation}
Indeed \eqref{eq:prvintconv1} is essentially a consequence of the fractional Leibniz Rule; see \cite[Theorem 7.1]{LS18}. Since $\eta \in C_c^\infty(\Omega)$, we have 
\[
 \|\laps{t} \eta \|_{L^r(\R^N)} = C(\eta,t,r) < \infty, \quad \forall t \geq 0, r \in (1,\infty).
\]
Thus, the fractional Leibniz Rule together with Sobolev embedding, \Cref{la:sobolev}, implies that for any $t < s$ with $N-(s-t)p > 0$,
\[
\begin{aligned}
\|\laps{s} & \brac{\eta \varphi} - \eta \laps{s} \varphi - \varphi \laps{s} \eta\|_{L^p(\R^N)}\\
 \quad\leq& C(\eta,p,s,t)\, \|\laps{t}\varphi\|_{L^{\frac{Np}{N-(s-t)p}}(\R^N)} 
\leq C(\eta,p,s,t)\, \|\laps{s}\varphi\|_{L^{p}(\R^N)} . 
 \end{aligned}
\]
From this we readily obtain \eqref{eq:prvintconv1}.

It follows from \eqref{eq:prvintconv1},  that if $\varphi \in H^{s,p}(\R^N)$ then $\eta \varphi \in H^{s,p}_{00}(\Omega)$, with the estimate
\[
 \|\eta \varphi \|_{H^{s,p'}_{00}(\Omega)} \equiv \|\eta \varphi \|_{L^{p'}(\R^N)} + \|\laps{s} (\eta \varphi) \|_{L^{p'}(\R^N)}\leq C(\eta) \|\varphi\|_{H^{s,p'}(\R^N)}.
\]
Given that $$|\tilde{h}_n[\varphi] |= |h_n[\eta \varphi]|\leq \|h_n\|_{{(H_{00}^{s,p'}(\Omega))^\ast}}\|\eta \varphi \|_{H^{s,p'}_{00}(\Omega)},$$
we observe 
\[
 \sup_{n} \|\tilde{h}_n\|_{(H^{s,p'}(\R^N))^\ast} \leq C(\eta)\, \sup_n \|h_n\|_{{(H_{00}^{s,p'}(\Omega))^\ast}} < \infty.
\]
In particular we have that $\tilde{h}_n$ is uniformly bounded in $H^{s,p'}(\R^N)^\ast$, and then $\tilde{h}_n\rightharpoonup\tilde{h}$ in $H^{s,p'}(\R^N)^\ast$ up to taking a subsequence.

By \Cref{pr:dualcharglobal}, there are representatives $f_n \in L^p(\R^N)$, $\laps{s}g_n \in L^p(\R^N)$  such that 
\[
 \tilde{h}_n[\varphi] = \int_{\R^N} f_n \varphi + \int_{\R^N} \laps{s} g_n\, \laps{s} \varphi \qquad \forall \varphi \in C_c^\infty(\R^N).
\]
Moreover,
\[
 \sup_{n} \|f_n\|_{L^p(\R^N)} + \sup_{n} \|g_n\|_{H^{s,p}(\R^N)}\aleq C(\eta)\, \sup_n \|h_n\|_{{(H_{00}^{s,p'}(\Omega))^\ast}} < \infty.
\]
By reflexivity of the incumbent spaces, up to taking a subsequence, we have the following weak convergences
\[
 f_n  \rightharpoonup f \quad \text{ in $L^p(\R^N)$},
\]
\[
 g_n  \rightharpoonup g \quad \text{ in $H^{s,p}(\R^N)$},
\]
\[
 \laps{s} g_n  \rightharpoonup \laps{s} g \quad \text{ in $L^{p}(\R^N)$}.
\]
Since $\tilde{h}_n$ weakly converges to $\tilde{h}$, we have
\[
 \tilde{h}[\varphi] = \int_{\R^N} f \varphi + \int_{\R^N} \laps{s} g\, \laps{s} \varphi \qquad \forall \varphi \in C_c^\infty(\R^N).
\]
By the compactness result of Rellich--Kondrachov, \Cref{pr:rellich}, we have strong convergence
\[
\begin{split}
	 g_n &\rightarrow g \qquad\quad \text{ in $L^{p}_{loc}(\R^N)$}\\\label{eq:fn}
 \lapms{s} f_n &\rightarrow \lapms{s} f \qquad \text{ in $L^{p}_{loc}(\R^N)$},
\end{split}
\]
also along subsequences. Here, $\lapms{s} f$ is the Riesz potential defined in \Cref{s:prelims}.

Let $\Omega' \subset \R^N$ be an arbitrary open and bounded set, and let $q < p$.

We will show in the following that 
\begin{equation}\label{eq:int:goalf}
\sup_{\varphi \in C_c^\infty(\Omega'),\, \|\varphi\|_{H^{s,q'}(\R^N)} \leq 1} \int_{\R^N} (f_n - f) \varphi \xrightarrow{n \to \infty} 0
\end{equation}
and
\begin{equation}\label{eq:int:goalg}
\sup_{\varphi \in C_c^\infty(\Omega'),\, \|\varphi\|_{H^{s,q'}(\R^N)} \leq 1} \int_{\R^N} \laps{s} (g_n -g)\laps{s} \varphi \xrightarrow{n \to \infty} 0.
\end{equation}

Together with the density of $C_c^\infty(\Omega')$-functions in $H^{s,q'}_{00}(\Omega')$, \Cref{th:density}, the limits \eqref{eq:int:goalf} and \eqref{eq:int:goalg} readily imply \eqref{eq:interiorconv} up to a subsequence. As remarked above, by the uniqueness of the weak limit, this establishes \eqref{eq:interiorconv}.

\underline{Strong convergence of $f_n$: Proof of \eqref{eq:int:goalf}}. 

Let $B(0,R_1)$ be a ball of radius $R_1>0$, containing $\Omega'$, centered at the origin, and let $R > 20R_1$. 

We define now the bump functions $\tilde{\lambda}, \xi,$ and $\lambda_R$ as follows. Let $\tilde{\lambda} \in C_c^\infty(B(0,2))$, $\tilde{\lambda} \equiv 1$ in $B(0,1)$ be arbitrary. 
Set $\xi (\cdot):= \tilde{\lambda}(\cdot/(2R_1)) \in C_c^\infty(B(0,4R_1))$, $\xi \equiv 1$ in $B(0,2R_1)$, and $\lambda_R := \tilde{\lambda}(\cdot / R) \in C_c^\infty(B(0,2R))$, $\lambda_R \equiv 1$ in $B(0,2R)$. 

Let $\varphi \in C_c^\infty(\Omega')$, then
\begin{equation}\label{eq:fnconvergence:split1st}
\begin{split}
\int (f_n -f) \varphi 
 =&\int \xi(f_n -f) \varphi \\
 =&c\int \xi(f_n -f) \lapms{s}\laps{s} \varphi \\
 =&c\int \lapms{s}(\xi(f_n -f)) \laps{s} \varphi. \\
 \end{split}
\end{equation}
The first equation in \eqref{eq:fnconvergence:split1st} is because $\xi \varphi = \varphi$, since $\xi \equiv 1$ in $B(0,2R_1)$ and $\varphi \in C_c^\infty(\Omega')$. The second equation in \eqref{eq:fnconvergence:split1st} follows from the basic properties of the Riesz potential, \Cref{la:basicprops}, and that $\lapms{s}\laps{s} \varphi  = c\varphi$. The third equation in \eqref{eq:fnconvergence:split1st} is the ``integration by parts'' in \Cref{la:basicprops}.

Thus, we find
\begin{equation}\label{eq:fnconvergence:split}
\begin{split}
 &|\int (f_n -f) \varphi| \\
 \aleq& \left |\int \lambda_{R} \lapms{s} (\xi (f_n -f)) \laps{s} \varphi\right | +\left | \int (1-\lambda_{R}) \lapms{s} (\xi (f_n -f)) \laps{s} \varphi \right |. 
 \end{split}
\end{equation}

For the first term on the right hand side of \eqref{eq:fnconvergence:split} we have
\[
 \left |\int \lambda_{R} \lapms{s} (\xi (f_n -f)) \laps{s} \varphi \right |\aleq \|\laps{s} \varphi\|_{L^{q'}(\mathbb{R}^n)}\, \|\lambda_{R} \lapms{s} (\xi (f_n -f))\|_{L^q(B_R(\Omega'))}.
\]
Observe that $\xi(f_n -f)$ converges weakly to zero in $L^p(\R^N)$. Since $q < p$, by the compact support of $\xi$ we conclude that $\xi(f_n -f)$ converges weakly to zero in $L^q(\R^N)$. By Rellich's theorem, \Cref{pr:rellich}, we conclude that for every fixed $R > 0$
\begin{equation}\label{eq:fnconvergence:1}
 \limsup_{n \to \infty} \sup_{\varphi \in C_c^\infty(\Omega'); \|\varphi\|_{H^{s,q'}(\R^N)} \leq 1}  \left |\int \lambda_{R} \lapms{s} (\xi (f_n -f)) \laps{s} \varphi \right |=0.
\end{equation}

For the second term on the right hand side  of \eqref{eq:fnconvergence:split}, observe that whenever $q_1,q_2 \in (1,\infty)$ with $\frac{1}{q_1} + \frac{1}{q_2} = 1$ then
\[
  \Big |\int (1-\lambda_{R}) \lapms{s} (\xi (f_n -f)) \laps{s} \varphi  \Big|\aleq \|\lapms{s} (\xi(f_n-f))\|_{L^{q_1}(\R^N)}\, \|(1-\lambda_{R}) \laps{s} \varphi\|_{L^{q_2}{(\R^N)}}.
\]
Choose $q_1$ such that $q_1 < \frac{Np}{N-sp}$ (if $sp \geq N$, this condition becomes $q_1 < \infty$), and $q_1 > \frac{N}{N-s}$. From Sobolev embedding, here in the form of \Cref{la:sobolev}, we obtain
\[
 \|\lapms{s} (\xi(f_n-f))\|_{L^{q_1}(\R^N)} \aleq \|\xi(f_n-f)\|_{L^{\frac{N q_1}{N+sq_1}}(B(0,4R_1))} \aleq \|f_n-f\|_{L^p(\R^N)}.
\]
In particular, since $f_n  \rightharpoonup f$  in $L^p(\R^N)$, we have 
\begin{equation}\label{eq:Boundfnf}
	 \sup_{n}\|\lapms{s} (\xi(f_n-f))\|_{L^{q_1}(\R^N)} <+\infty.
\end{equation}

For any $R>0$ sufficiently large, by the disjoint support of $1-\lambda_R$ and $\varphi$, we have
\[
 |\brac{(1-\lambda_{R}) \laps{s} \varphi}(x)| \aleq \int_{|x-y|\geq  R} |x-y|^{-N-s} |\varphi(y)|\, dy.
\]
Combining Young's inequality for convolutions, $\frac{1}{q_1} + \frac{1}{q_2} = 1$,  the fact that the support of $\varphi$ is in $\Omega'$, and H\"{o}lder's inequality, we  find
\begin{align*}
 \|(1-\lambda_{R}) \laps{s} \varphi\|_{L^{q_2}{(\R^N)}} &\aleq \left(\int \left(\int_{|x-y|\geq  R} |x-y|^{-N-s} dy\right)^{q_2}d x\right)^{1/q_2}\|\varphi\|_{L^1(\R^N)}\\
 &\aleq R^{(-N-s) +\frac{N}{q_2}} \|\varphi\|_{L^1(\R^N)}\\
 &\aleq
 R^{-s -\frac{N}{q_1}} \|\varphi\|_{L^1(\R^N)} \\
 &\aleq C(\Omega')\, R^{-s-\frac{N}{q_1}}\, \|\varphi\|_{L^{q'}(\Omega')},
\end{align*}
with $C(\Omega'):=|\Omega'|^{1/q}$.

That is, for all sufficiently large $R > 0$, we have for some fixed $\sigma > 0$ that
\begin{equation}\label{eq:fnconvergence:2}
\sup_{n} \sup_{\varphi \in C_c^\infty(\Omega'); \|\varphi\|_{H^{s,q'}(\R^N)} \leq 1} \left |\int (1-\lambda_{R}) \lapms{s} (\xi (f_n -f)) \laps{s} \varphi \right |\aleq C(\Omega') R^{-\sigma}.
\end{equation}
The convergence of $f_n$, namely \eqref{eq:int:goalf}, follows now from \eqref{eq:fnconvergence:split} together with \eqref{eq:fnconvergence:2} and \eqref{eq:fnconvergence:1} { by} taking first $n \to \infty$ and then letting $R \to \infty$.

\underline{Convergence of $g_n$: Proof of \eqref{eq:int:goalg}}
{
Let $K_3 \subsubset \R^N$ with $\Omega' \subset \subset K_3$. Take again a bump function $\xi \in C_c^\infty(K_3)$ with $\xi \equiv 1$ in a neighborhood of $\Omega'$.

Then for any $\varphi \in C_c^\infty(\Omega')$
\[
\begin{split}
 &\int_{\R^N} \brac{ \laps{s} g_n - \laps{s} g} \laps{s} \varphi\\
&= \int_{\R^N} \brac{ \laps{s} g_n - \laps{s} g} \xi \laps{s} \varphi \\
&\quad+\int_{\R^N} \brac{ \laps{s} g_n - \laps{s} g} (1-\xi) \laps{s} \varphi\\
&= \int_{\R^N} \brac{ \laps{s} g_n - \laps{s} g} \xi \laps{s} \varphi+\int_{\R^N} \brac{ g_n - g} \laps{s} \brac{(1-\xi) \laps{s} \varphi} . 
 \end{split}
\]
By the disjoint support of $\varphi$ and $(1-\xi)$ we can argue as for $f$, \eqref{eq:fnconvergence:split1st}, to conclude that 
\[
 \sup_{\varphi \in C_c^\infty(\Omega'),\, \|\varphi\|_{H^{s,q'}(\R^N)} \leq 1} \int_{\R^N}  \brac{ g_n - g} \laps{s} \brac{(1-\xi) \laps{s} \varphi} \xrightarrow{n \to \infty} 0.
\]
Thus, in order to establish \eqref{eq:int:goalg} it remains to show
\begin{equation}\label{eq:int:goalg2}
 \|\laps{s} g_n - \laps{s} g\|_{L^q(K_3)} \xrightarrow{n \to \infty} 0,
\end{equation}
which we do in the following.}

Since $\tilde{h}_n$ is a positive distribution, by \Cref{measurefact}, there exist {non-negative} Radon measures $\mu_n$ such that \[
 \int \varphi\, d\mu_n = \tilde{h}_n[\varphi] = \int_{\R^N} f_n \varphi + \int_{\R^N} \laps{s} g_n \laps{s} \varphi \qquad \forall \varphi \in C_c^\infty(\R^N).
\]
Observe that we can rewrite this as the distributional equation
\begin{equation}\label{eq:pdegn1}
 \laps{2s}g_n = \mu_n - f_n \quad \mbox{in } \R^N.
\end{equation}
We can mollify the PDE \eqref{eq:pdegn1} with the usual convolution kernel $\nu_\eps=\eps^{-n} \nu(\cdot/\eps)$ where $\nu \in C_c^\infty(B_1(0))$, $\nu \geq 0$ everywhere, normalized so that $\int \nu = \int \nu_\eps= 1$. Then,
\begin{equation}\label{eq:pdegn2}
 \laps{2s}(g_n\ast \nu_\eps) = \mu_n\ast \nu_\eps - f_n\ast \nu_\eps \quad \mbox{in } \R^N.
\end{equation}
Since the cutoff function $\eta$ used in the definition of $\tilde{h}_n$ clearly belongs to $H^{s,p'}_{00}(\Omega)$, we find that we can test $\tilde{h}_n$ with the constant function $1$, i.e. $\tilde{h}_n[1]$ is well defined. Then,
\[
 \sup_{n} \mu_n(\R^N) = \sup_{n} |\tilde{h}_n[1]| = \sup_{n} |h_n[\eta]| \leq \sup_n \|h_n\|_{(H^{s,p'}_{00}(\Omega))^\ast}\, \|\eta\|_{H^{s,p'}(\R^N)} < \infty.
\]
Thus, we obtain from Fubini's theorem
\[
 \sup_{\eps > 0} \sup_{n} \|\mu_n \ast \nu_\eps\|_{L^1(\R^N)} < \infty.
\]
Since moroever $f_n$ is bounded in $L^p(\R^N)$, for any bounded open set $K$ the sequence $f_n$ is bounded in $L^1(K)$. 

Take $K$ a bounded open set but large enough such that $\Omega' \subsubset K_3 \subsubset K$ where $K_3$ is from above, \eqref{eq:int:goalg2}.

Then we have
\[
 \sup_{\eps \in (0,1)} \sup_{n} \|f_n \ast \nu_\eps\|_{L^1(K)} < C(K) < \infty.
\]
That is, from \eqref{eq:pdegn2}, we have
\[
 \sup_{\eps \in (0,1)} \sup_{n}  \|\laps{2s}(g_n\ast \nu_\eps) \|_{L^1(K)} < C(K) < \infty.
\]
Moreover, since $\sup_{n} \|\laps{s}g_n\|_{L^p(\R^N)} < \infty$ we find that in particular,
\[
 \sup_{\eps \in (0,1)} \sup_{n}  \|\laps{s}(g_n\ast \nu_\eps) \|_{L^p(\R^N)} < \infty.
\]
Take an open set $K_1 \subsubset K$ such that $K_3 \subsubset K_1$. 
From \Cref{pr:sob2} we find that for any $\delta > 0$, $r \in (1,\infty)$ such that $s+\delta < 2s$, $s+\delta < n$ and $1 < r < \frac{N}{N-(s+\delta)}$,
\[
 \sup_{n} \sup_{\eps \in (0,1)} \|\laps{s+\delta} (g_n\ast \nu_\eps)\|_{L^r(K_1)} < C(K_1,K) < \infty.
\]
Taking $K_2 \subsubset K_1$ an open set such that $K_3 \subsubset K_2$
By weak compactness, letting $\eps \to 0$ we thus find 
\[
 \sup_{n} \|\laps{s+\delta} g_n\|_{L^r(K_2)} <  \infty.
\]
By Rellich's theorem in the form of \Cref{pr:rellich} (recall that $\sup_{n} \|g_n\|_{H^{s,p}(\R^N)} < \infty$), we then find that $\laps{s} g_n$ is strongly convergent in $L^r(K_3)$ for any $1 < r < \frac{N}{N-(s+\delta)}$. 
In particular we have that $\laps{s} g_n$ is strongly convergent in $L^1(K_3)$. Since $\laps{s} g_n$ is on the other hand weakly convergent to $\laps{s} g$ in $L^p(\R^N)$, by Vitali's theorem in form of \Cref{la:conv}, we find that $\|\laps{s} g_n-\laps{s} g\|_{L^q(K_3)} \xrightarrow{n \to \infty} 0$ whenever $1 \leq q < p$. That is  \eqref{eq:int:goalg2} is established. The proof is concluded.
\end{proof}

\subsection{Up to the boundary}
Once we have \Cref{th:convergenceint} which treats convergence in the interior of $\Omega$, the convergence up to the boundary is comparatively simpler. Further, note that although Theorem \ref{th:convergenceint} is proven for a general $p$, in the following result we consider only the $p=2$ case.

\begin{theorem}\label{th:uptobd}
Let $q \in (1,2)$, $s \in (0,1)$, and  $h_n, h \in H^{s,2}_{00}(\Omega)^\ast$ be such that 
\[
 \sup_{n} \|h_n\|_{H^{s,2}_{00}(\Omega)^\ast} < \infty,
\]
and assume
\begin{itemize}
 \item we have \emph{locally} strong convergence of $h_n$ to $h$ in {$H^{s,q'}_{00}(\Omega)^\ast$} in the following sense: for all $\eta \in C_c^\infty(\Omega)$ we have
\[
 \|\tilde{h}_{n,\eta}-\tilde{h}_{\eta}\|_{H^{s,q'}_{00}(\Omega)^\ast} \xrightarrow{k \to \infty} 0,
\]
where
\[
 \tilde{h}_{n,\eta}[\varphi] := h_n[\eta \varphi], \quad \text{and}\quad  \tilde{h}_{\eta}[\varphi] := h[\eta \varphi].
\]
\end{itemize}
Then we have \emph{global} strong convergence of $h_n$ to $h$ in $(H^{s,q'}_{00}(\Omega))^\ast$, namely
\[
 \|\tilde{h}_n-\tilde{h}\|_{H^{s,q'}_{00}(\Omega)^\ast} \xrightarrow{k \to \infty} 0.
\]
\end{theorem}

\begin{proof}[Proof of \Cref{th:uptobd}]
\underline{Part 0: Preliminaries on weak convergence.}

Set 
\[
 \Lambda := \sup_{n} \|h_n\|_{H^{s,2}_{00}(\Omega)^\ast} + \|h\|_{H^{s,2}_{00}(\Omega)^\ast} < \infty.
\]
From \Cref{pr:id} we obtain $u_n, u \in H^{s,2}_{00}(\Omega)$ such that 
\[
h_n[\varphi] = \int_{\R^N} \laps{s} u_n\, \laps{s} \varphi \qquad \forall \varphi \in C_c^\infty(\Omega)
\]
and
\[
h[\varphi] = \int_{\R^N} \laps{s} u\, \laps{s} \varphi \qquad \forall \varphi \in C_c^\infty(\Omega).
\]
Moreover we have the estimate
\[
\sup_{n} \|\laps{s} u_n\|_{L^2(\R^N)} +\|\laps{s} u\|_{L^2(\R^N)} \aleq \Lambda.
\]
In particular $\laps{s} u_n$ has a weakly converging subsequence with respect to $L^2(\R^N)$ to a $v \in L^2(\R^N)$. We claim 
\begin{equation}\label{eq:vequ2334}
 v = \laps{s} u.
\end{equation}
Indeed, by the weak convergence $h_n \to h$ we obtain that
\begin{equation}\label{eq:vmlapsu:22}
 \int_{\R^N} \brac{v- \laps{s} u}\, \laps{s} \varphi = 0 \quad \forall \varphi \in C_c^\infty(\R^N).
\end{equation}
Since both $v, \laps{s} u \in L^2(\R^N)$, by duality and density of $C_c^\infty(\R^N)$ in $L^2(\R^N)$ there exists $\psi \in C_c^\infty(\R^N)$ with $\|\psi\|_{L^2(\R^N)} \leq 1$ such that 
\[
 \|v-\laps{s} u\|_{L^2(\R^N)} \leq 2 \int_{\R^N} (v-\laps{s} u) \psi.
\]
Take $\rho_0 > 0$ such that $B(0,\rho_0) \supset \supp \psi$.

Let $\eta \in C_c^\infty(B(0,2))$ with $\eta \equiv 1$ in $B(0,1)$, and set $\chi_\rho := \eta(x/\rho)$. Then the fact that $\psi = \laps{s} \lapms{s} \psi$, \Cref{la:basicprops}, and \eqref{eq:vmlapsu:22} imply
\[
\begin{split}
 \int_{\R^N} (v-\laps{s} u) \psi =& \int_{\R^N} (v-\laps{s} u) \laps{s} \lapms{s} \psi\\
 =& \int_{\R^N} (v-\laps{s} u) \laps{s} \brac{\chi_\rho \lapms{s} \psi} \\
 &\quad+ \int_{\R^N} (v-\laps{s} u) \laps{s} \brac{(1-\chi_\rho) \lapms{s} \psi}\\
 \overset{\eqref{eq:vmlapsu:22}}{=}& \int_{\R^N} (v-\laps{s} u) \laps{s} \brac{(1-\chi_\rho) \lapms{s} \psi}\\
 =& \int_{\R^N} (v-\laps{s} u) \chi_{2\rho}\laps{s} \brac{(1-\chi_\rho) \lapms{s} \psi}\\
 &+\int_{\R^N} (v-\laps{s} u) (1-\chi_{2\rho})\laps{s} \brac{(1-\chi_\rho) \lapms{s} \psi}.
 \end{split}
\]
By \cite[Proposition 2.4.]{MSY20} for any $\rho > \rho_0$,
\[
 \|(1-\chi_{2\rho})\laps{s} \brac{(1-\chi_\rho) \lapms{s} \psi}\|_{L^2(\R^N)} \aleq \|\psi\|_{L^2(\R^N)} \leq 1.
\]
and there exists $p < 2$ such that 
\[
 \|\chi_{2\rho}\laps{s} \brac{(1-\chi_\rho) \lapms{s} \psi}\|_{L^2(\R^N)} \aleq \rho^{\frac{n}{p}-\frac{n}{2}}\|\psi\|_{L^p(\R^N)} \leq \rho^{\frac{n}{p}-\frac{n}{2}}.
\]
Observe that the dependency with respect to $\rho$ can be obtained from \cite[Proposition 2.4.]{MSY20} by scaling.
So that 
\[
\begin{split}
 &\left |\int_{\R^N} (v-\laps{s} u) (1-\chi_{2\rho})\laps{s} \brac{(1-\chi_\rho) \lapms{s} \psi} \right |\\
 \aleq &\|v-\laps{s} u\|_{L^2(\R^N \backslash B(2\rho))} \xrightarrow{\rho \to \infty} 0.
 \end{split}
\]
In the last step we used that $v-\laps{s} u\in L^2(\R^N)$ and absolute continuity of the integral.

Moreover, 
\[
 \left |  \int_{\R^N} (v-\laps{s} u) \chi_{2\rho}\laps{s} \brac{(1-\chi_\rho) \lapms{s} \psi}\right | \aleq \|v-\laps{s} u\|_{L^2(\R^N)} \rho^{\frac{n}{p}-\frac{n}{2}} \xrightarrow{\rho \to \infty} 0.
\]
This implies that 
\[
 \|v-\laps{s} u\|_{L^2(\R^N)} = 0.
\]
That is, \eqref{eq:vequ2334} is established and we have
\[
 \laps{s} u_n  \rightharpoonup \laps{s} u \quad \text{weakly in $L^2(\R^N)$}.
\]
\underline{Part 1: strong convergence of $u_n$}

From Poincar\'e inequality (observe that $u_n, u \equiv 0$ in $\R^N \backslash \Omega$) we have 
\[
 \sup_{n \in \N} \|u_n-u\|_{L^2(\R^N)} \aleq \sup_{n \in \N}\|\laps{s} (u_n-u)\|_{L^2(\R^N)} < \infty.
\]
Since $\laps{s} (u_n - u)  \rightharpoonup 0$ weakly in $L^2(\R^N)$ and $u_n - u$ is uniformly bounded in $H^{s,2}(\R^N)$ we obtain from Rellich's theorem, \Cref{pr:rellich}, $u_n - u \xrightarrow{n \to \infty} 0$ strongly in $L^2_{loc}(\R^N)$. Since $u_n = u = 0$ in $\R^N \backslash \Omega$ this implies $u_n - u \xrightarrow{n \to \infty} 0$ strongly in $L^2(\R^N)$. Using yet again the compact support of $u_n-u$, we find
\begin{equation}\label{eq:rellich}
 \|u_n-u\|_{L^q(\R^N)} \leq C(\Omega) \|u_n-u\|_{L^2(\R^N)} \xrightarrow{n \to \infty} 0.
\end{equation}

\underline{Part 2: strong convergence of $\laps{s} u_n$}
It remains to prove
\begin{equation}\label{eq:lapsunstrongconv}
 \|\laps{s} u_n - \laps{s} u\|_{L^q(\R^N)} \xrightarrow{n \to \infty} 0.
\end{equation}
{We will reduce the proof of \eqref{eq:lapsunstrongconv} to estimates in three subsets of $\R^n$
\begin{itemize}
 \item the ``outer part'' $\Omega_{o,\eps}$, a subset of $\R^n \backslash \Omega$ with positive distance to $\partial \Omega$
 \item the ``inner part'' $\Omega_{i,\eps}$, a subset of $\Omega$ with positive distance to $\partial \Omega$
 \item the ``close to the boundary part'' $\Omega_{\partial, \eps}$, the set of points close to the boundary $\partial \Omega$.
\end{itemize}
}
{\underline{Step 0: Reducing to three regimes.} Our goal of this preliminary step is to reduce the situation to the estimate of the three sets mentioned above, namely \eqref{eq:sc:split} below.}

Let $\eps \in (0,1)$. Since $\Omega$ is open and bounded, the set $\partial \Omega$ is compact. Set 
\[
 B_\eps(\partial \Omega) := \{x \in \R^N: \dist(x,\partial \Omega) < \eps\}.
\]
Since $\Omega$ is bounded, $B_\eps(\partial \Omega)$ is bounded. Moreover, by the dominated convergence theorem\footnote{This is the only place in this entire work where we use that $\partial \Omega$ is Lipschitz.}
\[
 |B_\eps(\partial \Omega)| \xrightarrow{\eps \to 0}  0.
\]

We call $B_\eps(\partial \Omega)$ the boundary regime $\Omega_{\partial,\eps} := B_{\eps} (\partial \Omega)$.
Moreover we define the outer part $\Omega_{o,\eps} := \R^N \backslash B_{\eps} (\Omega)$ and the interior part $\Omega_{i,\eps} := \Omega \backslash \Omega_{\partial, \eps}$.
Then 
\begin{equation}\label{eq:sc:split}
\begin{split}
 \|\laps{s} u_n - \laps{s} u\|_{L^q(\R^N)} \leq& \|\laps{s} u_n - \laps{s} u\|_{L^q(\Omega_{o,\eps})}\\
 &+ \|\laps{s} u_n - \laps{s} u\|_{L^q(\Omega_{\partial,\eps})} \\
 &+ \|\laps{s} u_n - \laps{s} u\|_{L^q(\Omega_{i,\eps})}.
 \end{split}
\end{equation}
\underline{Step 1:} first we treat for the \underline{boundary part}. Since $q < 2$,
\begin{equation}\label{eq:sc:split:1}
\|\laps{s} u_n - \laps{s} u\|_{L^q(\Omega_{\partial,\eps})} \leq |\Omega_{\partial,\eps}|^{\frac{1}{q}-\frac{1}{2}}   \Lambda \xrightarrow{\eps \to 0} 0.
\end{equation}
\underline{Step 2:} Secondly, we treat the \underline{outer part}:
Since $u_n-u$ vanishes in $\R^N \backslash \Omega$, for $x \in \Omega_{o,\eps}$ we have
\[
 \laps{s} (u_n-u)(x) = c\int_{\R^N} \underbrace{|x-y|^{-s-n}}_{\ageq \eps^{-s-n}} (u_n(y)-u(y))\, dy,
\]
so by Young's estimate for convolutions,
\begin{equation}\label{eq:sc:split:2}
\begin{split}
 \|\laps{s} (u_n-u)\|_{L^q(\Omega_{o,\eps})} \aleq &\|u_n - u\|_{L^q(\R^N)}\, \int_{|z| > \eps} |z|^{-n-s}\, dz\\
 =&C\, \|u_n - u\|_{L^q(\R^N)}\, \eps^{-s}\\
\end{split}
\end{equation}
\underline{Step 3:} Lastly for the \underline{inner part} we use the assumption of local convergence. For our fixed $\eps > 0$ let $\eta_\eps \in C_c^\infty(\Omega)$ with $\eta_\eps \equiv 1$ in $\Omega_{i,\eps/2}$ (i.e. in an $\eps$-neighborhood of $\Omega_{i,\eps})$. Then, by duality, for any $n \in \N$ there exist $F_n \in C_c^\infty(\Omega_{i,\eps})$ with $\|F_n\|_{L^{q'}(\R^N)} \leq 2$ such that
\[
 \|\laps{s} (u_n-u)\|_{L^q(\Omega_{i,\eps})} \leq \int_{\R^N} \laps{s} (u_n-u)\, F_n.
\]
Since $F_n$ is smooth we may write ($\lapms{s} = (-\lap)^{-\frac{s}{2}}$ is the Riesz potential)
\[
 F_n = \laps{s} \lapms{s} F_n = \laps{s} \brac{(\eta_\eps)^2 \lapms{s} F_n} +\laps{s} \brac{(1-(\eta_\eps)^2) \lapms{s} F_n}
\]
Let us use the commutator notation,
\[
 [\laps{s} ,f] (g) := \laps{s} (fg) -f\laps{s} g.
\]
Using again $\lapms{s} \laps{s} g= g$, \Cref{la:basicprops}, we have
\[
 \laps{s} \brac{\eta_\eps \lapms{s} F_n} = \eta_\eps F_n + [\laps{s} ,\eta_\eps](\lapms{s} F_n).
\]
By Coifman--McIntosh--Meyer type commutator estimates, here in the form of \cite[Theorem 6.1.]{LS18}, we observe
\[
 \begin{split}
 & \|[\laps{s} ,\eta_\eps] (\lapms{s} F_n)\|_{L^{q'}(\R^N)}\\
 \aleq&\|\eta_\eps F_n\|_{L^{q'}(\R^N)} + \brac{1+\|\nabla \eta_\eps\|_{L^\infty(\R^N)}}\, \|\lapms{s} F_n\|_{L^{q'}(\R^N)}
 \aleq C(\eps) \|F_n\|_{L^{q'}(\R^N)}
 \end{split}
\]
where in the last step we used the estimate\footnote{this estimate holds whenever $\frac{Nq'}{N+sq'} > 1$, equivalently $\frac{N}{q} > s$. When $N \geq 2$, $q < 2$ this is true whenever $s \in (0,1)$. For $N=1$, $q$ needs to be taken small enough. But this is not a problem since again, strong convergence at $L^{q_0}$ ($q_0 > 1$ small) and weak convergence at $L^2$ implies strong convergence for any $L^q$, $q \in (1,2)$ by Vitali's theorem.}
\[
 \|\lapms{s} F_n\|_{L^{q'}(\R^N)} \aleq \|F_n\|_{L^{\frac{Nq'}{N+sq'}}(\R^N)} = \|F_n\|_{L^{\frac{nq'}{n+sq'}}(\Omega_{i,\eps})} \aleq \|F_n\|_{L^{q'}(\R^N)}
\]
That is
\[
\sup_{n} \|\laps{s} \brac{\eta_\eps \lapms{s} F_n}\|_{L^{q'}(\R^N)} \leq 2\, C(\eps,\Omega).
\]
By the relationship of $u_n, u$ to $h_n$, $h$ we have
\[
\begin{split}
 &\int_{\R^N} \laps{s} (u_n-u)\, \laps{s} \brac{(\eta_\eps)^2 \lapms{s} F_n} \\
 =& (h_n-h)[\eta_\eps \brac{\eta_\eps \lapms{s} F_n}] = (\tilde{h}_{n,\eta_\eps}-\tilde{h}_{\eta_\eps})[\eta_\eps \lapms{s} F_n] 
\end{split}
 \]
and thus for any fixed $\eps > 0$, by the local strong convergence of $h_n$ to $h$
\[
 \int_{\R^N} \laps{s} (u_n-u)\, \laps{s} \brac{(\eta_\eps)^2 \lapms{s} F_n} \aleq \|\tilde{h}_{n,\eta_\eps}-\tilde{h}_{n,\eta_\eps}\|_{H^{s,q'}_{00}(\Omega)^\ast} \xrightarrow{n \to \infty} 0.
\]
It remains to estimate 
\[
\begin{split}
 &\int_{\R^N} \laps{s} (u_n-u)\, \laps{s} \brac{(1-(\eta_\eps)^2) \lapms{s} F_n}\\
 =&\int_{\R^N} (u_n-u)\, \laps{2s} \brac{(1-(\eta_\eps)^2) \lapms{s} F_n}\\
 =&\int_{\Omega} (u_n-u)\, \laps{2s} \brac{(1-(\eta_\eps)^2) \lapms{s} F_n}\\
 \aleq &\|u_n-u\|_{L^2(\R^N)} \left \|\laps{2s} \brac{(1-(\eta_\eps)^2) \lapms{s} F_n} \right \|_{L^2(\Omega)}.
\end{split}
 \]
Set $\chi_\eps := 1-(\eta_\eps)^2$. That is,
\[
G(x) :=  \brac{(1-(\eta_\eps)^2) \lapms{s} F_n }(x)= \int_{\R^N} |x-y|^{s-N} \chi_\eps(x)\, F_n(y)\, dy.
\]
Observe that $\chi_\eps$ and $F_n$ have disjoint support, so we may choose yet another another cutoff function $\tilde{\eta}_\eps \in C_c^\infty(\Omega_{i,\frac{3}{4} \eps})$ with $\tilde{\eta}_{\eps} \equiv 1$ in $\Omega_{i,\eps}$ (that is: $F_n = F_n \tilde{\eta}_\eps$. Then we have for
\[
 k(x,y) := |x-y|^{s-N} \chi_\eps(x)\, \tilde{\eta}_{\eps}(y)
\]
that 
\[
 G(x) = \int_{\R^N} k(x,y)\, F_n(y)\, dy.
\]
Moreover, since $|z|^{s-N} \in L^p(\R^N \backslash B_r(0))$ for any $r > 0$ and $p>\frac{N}{N-s}$,  we find by Young's inequality for integral operators that 
\[
 \|G\|_{L^p(\R^N)} \leq \|F_n\|_{L^1(\R^N)} \leq C(\Omega,\eps)\|F_n\|_{L^2(\Omega)} \leq 2C(\Omega,\eps).
\]
Moreover, $k(x,y)$ is smooth and computing the second derivative with respect to $x$ (notice $\nabla \chi_\eps(x)$ has compact support) we find that (for $s \in (0,2)$)
\[
 \frac{\partial^2}{\partial{x_j}\partial{x_i}} k(x,y) \in L^1 \cap L^\infty(\R^N \times \R^N), \quad\text{ for all }\quad 1\leq i,j\leq N. 
\]
So we have 
\[
 \left\|\frac{\partial^2G}{\partial{x_j}\partial{x_i}} \right\|_{L^p(\R^N)} \aleq \|F_n\|_{L^1(\R^N)} \aleq 2C(\Omega,\eps),\quad\text{ for all }\quad 1\leq i,j\leq N.
\]
That is $G \in W^{2,p}(\R^N) \hookrightarrow H^{s,p}(\R^N)$ (Sobolev embedding) for any $s \in (0,2)$. In particular,
\[
 \|\laps{2s} G\|_{L^p(\R^N)} \aleq 2C(\Omega,\eps).
\]
Together we have shown for the interior case
\begin{equation}\label{eq:sc:split:3}
 \|\laps{s} (u_n-u)\|_{L^q(\Omega_{i,\eps})} \leq C(\Omega,\eps)\, \brac{\|u_n-u\|_{L^2(\Omega)} + \|\tilde{h}_{n,\eta_\eps}-\tilde{h}_{n,\eta_\eps}\|_{H^{s,q'}_{00}(\Omega)^\ast}}
\end{equation}

Plugging this all together, namely \eqref{eq:sc:split:1}, \eqref{eq:sc:split:2}, \eqref{eq:sc:split:3} into \eqref{eq:sc:split}, we have shown that for any $\eps > 0$,
\[
\begin{split}
 \|\laps{s} u_n - \laps{s} u\|_{L^q(\R^N)} \leq C(\Omega)\, \|u_n - u\|_{L^q(\R^N)}\, \eps^{-s}+ C(\Omega)\, \eps^{\frac{1}{q}-\frac{1}{2}}  \Lambda\\
 + C(\Omega,\eps)\, \brac{\|u_n-u\|_{L^2(\Omega)} + \|\tilde{h}_{n,\eta_\eps}-\tilde{h}_{n,\eta_\eps}\|_{H^{s,q'}_{00}(\Omega)^\ast}}
 \end{split}
\]
In view of the locally strong convergence of $h$, Rellich as in \eqref{eq:rellich}, we let $n \to \infty$ and obtain that for any $\eps > 0$,
\[
\begin{split}
 \lim_{n \to \infty} \|\laps{s} u_n - \laps{s} u\|_{L^q(\R^N)} \leq& C(\Omega)\, \eps^{\frac{1}{q}-\frac{1}{2}}  \Lambda\\
 \end{split}
\]
This holds for any $\eps > 0$ and thus for $\eps \to 0$ (since $q < 2$),
\[
 \lim_{n \to \infty} \|\laps{s} u_n - \laps{s} u\|_{L^q(\R^N)} =0.
\]
\eqref{eq:lapsunstrongconv} is established, and the proof of \Cref{th:uptobd} is complete. 
\end{proof}

\section{Mosco Convergence in the Fractional Sobolev Setting}\label{sec:Mosco}

Mosco convergence \cite{mosco1969convergence,mosco1967approximation} is a set convergence concept with useful application in several areas of mathematics. The definition is 
appropriate to the study of problems involving moving convex sets in reflexive Banach spaces, as we now explain. In particular, it has made possible to study stability properties of variational inequalities and associated problems (see \cite{MenaldiR2021} for a modern account\color{black}). One class of those associated problems are the quasi-variational inequalities (QVIs), that can be thought as variational inequalities where the constraint is implicit (in the sense that it is state dependent). QVIs arose initially from the work of Bensoussan and Lions \cite{Bensoussan1974,Lions1973} (see also the monographs \cite{MR673169,Ben1982,Bensoussan1984}) on impulse control problems, and later found application modeling a wide variety of non-convex and non-smooth phenomena in applied sciences. Specifically, areas including superconductivity (\cite{Rodrigues2000,MR1765540,MR2947539,MR2652615,MR3335194,Prigozhin,MR3023771,hintermuller2019dissipative}), continuum mechanics (\cite{Friedman1982,Rodrigues2000,AHR,alphonse2019recent}), 
growth of sandpiles and the determination of rivers/lakes networks (\cite{MR3082292,MR3231973,MR3335194,Prigozhin2012,Prigozhin1986,Prigozhin1994,Prigozhin1996,MR3231973,Prigozhin1994,Prigozhin1996}), 
among others. Further,  a first account on problems with fractional-gradient constraints  can be found  on \cite{rodrigues2019nonlocal}, and  \color{black}a concrete application associated to fractional spaces can be found on~\cite{antil2017fractional}. For a complete and classical account on QVIs, we refer the reader to the text of Baiocchi and Capelo \cite{BaiC1984}.

We now provide an application of the compact embedding results established in the previous section to the convergence of convex sets. This, as stated in the introduction of the paper, can be seen as the extension of the Boccardo/Murat result for fractional Sobolev spaces. We start by introducing the definition of Mosco convergence \cite{mosco1969convergence}:

\begin{definition}[\textsc{Mosco convergence}]\label{definition:MoscoConvergence}
Let $K$ and $K_n$, for each $n\in\mathbb{N}$, be non-empty, closed and convex subsets of  a reflexive Banach space $V$. Then the sequence \textit{$\{K_n\}$ {is said to} converge to $K$ in the sense of Mosco} as $n\rightarrow\infty$, {denoted} by $$K_n\Mosco K,$$ if the following two conditions {are fulfilled}:
\begin{enumerate}[\upshape(I)]
  \item\label{itm:1}  For each $w\in K$, there exists $\{w_{n'}\}$ such that $w_{n'}\in K_{n'}$ for $n'\in \mathbb{N}'\subset \mathbb{N}$ and $w_{n'}\rightarrow w$ in $V$.
  \item\label{itm:2} If $w_n\in K_n$ and $w_n\rightharpoonup w$ in $V$ along a subsequence, then $w\in K$.
\end{enumerate}
\end{definition} 

The notion of Mosco convergence provides the right notion of convergence to establish stability results for variational inequalities and constrained optimization problems. Specifically, we have the following (see \cite{mosco1969convergence} and \cite{MR880369} for a proof).

\begin{theorem}[Mosco]\label{thm:Mosco}
	Let $V$ be a reflexive Banach space, $f\in V'$, and $A:V\to V'$ be linear, bounded, and coercive. Suppose in addition that for each $n\in\mathbb{N}$, $K_n$ and $K$ are convex, closed, and non-empty subsets of $V$, and that
	\begin{equation*}
		K_n\Mosco K.
	\end{equation*}
	Then, the sequence of unique solutions $\{u_n\}$ to
	\begin{equation*}
		\text{ Find } u\in K_n \quad:\quad \langle Au-f,v-u\rangle \geq 0, \quad \text{for all } v\in K_n,
	\end{equation*}
	converges $V$-strongly to $u^*$, the unique solution to
		\begin{equation*}
		\text{ Find } u\in K \quad:\quad \langle Au-f,v-u\rangle \geq 0, \quad \text{for all } v\in K.
	\end{equation*}
\end{theorem}
It follows that in order to determine stability results for variational inequalities, conditions that guarantee Mosco convergence are of interest. While for some kind of sets, necessary and sufficient conditions are known (see \cite{MR791845}, the application of such to particular examples is an arduous task). On the other hand, there exist some useful (and simple to verify in applications) sufficient conditions such as one arising from the so-called Boccardo-Murat \cite{BM79} result: A direct consequence of the Murat/Brezis compactness.

In the classical setting, the result of Murat, namely that the positive cone of non-negative elements of $H^{-1}(\Omega)$ compactly embeds in $W^{-1,q}(\Omega)$ for $q<2$, implies that Mosco convergence of the sets 
\[
 K(\phi_n) := \left \{v \in H^{1}_{0}(\Omega): v \geq \phi_n \quad \text{a.e.} \right \}
\]
to the set $ K(\phi)$ is {achieved} not only if $\phi_n\to\phi$ in $H_0^1(\Omega)$, but also if $\phi_n\rightharpoonup \phi$ in $W_0^{1,p}(\Omega)$ for $p>2$. In our setting, we prove that our result stating that the positive cone of non-negative elements in $H_{00}^{s,2}(\Omega)^*$ compactly embeds into $H_{00}^{s,q}(\Omega)^*$ for $1\leq q<2$, implies that the sequence of sets 
\[
 K_s(\psi_n) := \left \{v \in H^{s,2}_{00}(\Omega):  v \geq \psi_n \quad \text{a.e.} \right \}
\]
Mosco converges to $K_s(\psi)$  provided that $\psi_n\rightharpoonup \psi$ in $H^{s,p}_{00}(\Omega)$ with $p>2$. This corresponds to the analogous result to the Boccardo-Murat \cite{BM79} one but in the fractional Sobolev framework.

Initially, note the following result for existence of solutions to the class of variational inequalities of interest. 

\begin{lemma}\label{la:existunique}
Let $f \in H^{s,2}_{00}(\Omega)^\ast$ and $\psi \in H^{s,2}_{00}(\Omega)$. There exists a unique $u \in K_s(\psi)$ such that
\begin{equation}\label{eq:qvi}
 \int_{\R^N} \laps{s} u\, \laps{s} (v-u) \geq f[v-u] \quad \forall v \in K_s(\psi).
\end{equation}
\end{lemma}
\begin{proof}
Existence and uniqueness follow from minimizing the strictly convex energy
\[
 \mathcal{E}(u) := \frac{1}{2} \|\laps{s} u\|_{L^2(\R^N)}^2 - f[u] \quad \text{in $K_s(\psi)$}.
\]
The Direct Method of Calculus of Variations is applicable. 
\end{proof}

The next statement is the analogue of \cite[Th\'eor\`eme 1]{BM79}
\begin{theorem}\label{th1}
Let $f \in H^{s,2}_{00}(\Omega)^\ast$ and {consider} a sequence of obstacles $\psi_n \in H^{s,q}_{00}(\Omega)$ weakly converging to $\psi$ in $H^{s,q}_{00}(\Omega)$  for some $q > 2$. Denote by $u_n$, and $u$ the solutions from \Cref{la:existunique} {corresponding to} $K_s(\psi_n)$ and $K_s(\psi)$, respectively. Then
\[
 u_n \xrightarrow{n \to \infty} u \quad \text{strongly in $H^{s,2}_{00}(\Omega)$}.
\]
\end{theorem}
\begin{proof}

{Given that $u$ is uniquely determined, it suffices to show the convergence of $u_n$ for some subsequence as $n \to \infty$.}

Testing the variational inequality \eqref{eq:qvi} for $u_n$ (when the constraint set is given by $ K_s(\psi_n)$) with $v = \psi_n$ we readily find
\[
 \|\laps{s} u_n\|_{L^2(\R^N)} \aleq \|\laps{s} \psi_n\|_{L^2(\R^N)} + \|f\|_{H^{s,2}_{00}(\Omega)^\ast}.
\]
In view of \eqref{eq:semi_norm} (up to a subsequence) we thus find that $u_n$ converges weakly to some $\tilde{u}$ in $H^{s,2}_{00}(\Omega)$, $\psi_n \to \psi$, $u_n \to \tilde{u}$ in $L^2(\Omega)$ and pointwise a.e., and $\laps{s} u_n \rightharpoonup \laps{s} \tilde{u}$ in $L^2(\R^N)$.  {Further, the weak lower semicontinuity property of the norm  implies}
\begin{equation}\label{eq:M2}
 \|\laps{s} \tilde{u}\|_{L^2(\R^N)}^2 \leq \liminf_{{n \to \infty}} \|\laps{s} u_n\|_{L^2(\R^N)}^2.
\end{equation}

On the other hand, $u_n-\psi_n \geq 0$ a.e., and since the subsequence of $\psi_n$ and $u_n$ converge a.e. we find in the limit that $\tilde{u} -\psi \geq 0$ a.e., that is $$\tilde{u} \in K_s(\psi).$$

Testing again the variational inequality \eqref{eq:qvi} for $u_n$ with now $u_n + \phi$ (for some nonnegative $\phi$) we find in the distributional sense that
\[
 \laps{2s} u_n - f \geq 0 \quad \forall n \in \N \quad \text{in $\Omega$}.
\]

Moreover since we have from the weak convergence of $u_n$ that 
\[
 \laps{2s} u_n - f \rightharpoonup \laps{2s} \tilde{u} - f \quad \text{in $H^{s,2}_{00}(\Omega)^\ast$},
\]
we observe
\[
 \laps{2s} \tilde{u} - f \geq 0 \quad \text{in $\Omega$} . 
\]

From \Cref{th:convergence}, we obtain that for any $p < 2$, 
\[
 \sup_{\|\varphi\|_{H^{s,p'}_{00}(\Omega)} \leq 1} \int_{\R^N} \left(\laps{2s} u_n -\laps{2s} \tilde{u} \right) \varphi \xrightarrow{{n \to \infty}} 0,
\]
where (as usual) the above integral is understood in the duality pairing sense. Then
\[
 \sup_{\|\varphi\|_{H^{s,p'}_{00}(\Omega)} \leq 1} \int_{\R^N} \left(\laps{s} u_n -\laps{s} \tilde{u} \right) \laps{s}\varphi \xrightarrow{{n \to \infty}} 0,
\]
and hence
\begin{equation}\label{eq:strongLpconvergence}
 \|\laps{s} u_n - \laps{s} \tilde{u}\|_{L^{p}(\R^N)} \xrightarrow{{n \to \infty}} 0.
\end{equation}
Note that we can conclude that $\laps{s} u_n \xrightarrow{{n \to \infty}} \laps{s} \tilde{u}$ in $L^2(\R^N)$ if we show
\begin{equation}\label{eq:normconvergence}
 \|\laps{s} u_n \|_{L^2(\R^N)} \xrightarrow{{n \to \infty}} \|\laps{s} \tilde{u}\|_{L^{2}(\R^N)} .
\end{equation}
To establish \eqref{eq:normconvergence}, given that \eqref{eq:M2} holds, it only remains to show
\[
 \limsup_{{n \to \infty}} \|\laps{s} u_n \|_{L^2(\R^N)} \leq \|\laps{s} \tilde{u}\|_{L^{2}(\R^N)}.
\]
For this we test the variational inequality \eqref{eq:qvi} for $u_n$ with $v := \psi_n + \tilde{u} - \psi$; observe that $\tilde{u}-\psi \geq 0$ since $\tilde{u} \in K(\psi)$. Then
\[
 \int_{\R^N} \laps{s} u_n \laps{s}(\psi_n + \tilde{u} - \psi-u_n ) \geq f[\psi_n + \tilde{u} - \psi-u_n ].
\]
i.e.
\[
\begin{split}
 \|\laps{s} u_n\|_{L^2(\R^N)}^2 \leq &\int_{\R^N} \laps{s} u_n \laps{s}\tilde{u}+\int_{\R^N} \laps{s} (u_n -\tilde{u})\laps{s}(\psi_n - \psi)\\ 
 &+\int_{\R^N} \laps{s} \tilde{u} \laps{s}(\psi_n - \psi) + f[\psi-\psi_n]+f[u_n-\tilde{u}].
 \end{split}
\]
By weak convergence $u_n$ to $u$ in $H^{s,2}_{00}(\Omega)$, we observe
\[
 \lim_{{n \to \infty}} \int_{\R^N} \laps{s} u_n \laps{s}\tilde{u} = \|\laps{s} \tilde{u}\|_{L^2(\R^N)}^2, 
\]
and by the weak convergence of $\psi_n$ to $\psi$ in $H^{s,q}_{00}(\Omega)$, we have
\[
 \lim_{{n \to \infty}} \left(f[\psi-\psi_n]+f[u_n-\tilde{u}] \right)= 0,
\]
and
\[
\lim_{{n \to \infty}} \int_{\R^N} \laps{s} \tilde{u} \laps{s}(\psi_n - \psi) =0.
\]
By the uniform boundedness of $\laps{s} (\psi_n-\psi)$ in $L^q(\R^N)$ (recall that $q>2$) and the strong convergence of $\laps{s} u_n$ in $L^{q'}(\R^N)$,
\begin{align*}
&\left |\int_{\R^N} \laps{s} (u_n -\tilde{u})\laps{s}(\psi_n - \psi) \right |\aleq \\
&\qquad\qquad\|\laps{s} u_n - \laps{s} \tilde{u}\|_{L^{q'}(\R^N)}\, \|\laps{s}(\psi_n - \psi)\|_{L^q(\R^N)} \xrightarrow{{n \to \infty}} 0.	
\end{align*}

In conclusion,
\[
 \limsup_{{n \to \infty}} \|\laps{s} u_n\|_{L^2(\R^N)}^2 \leq \int_{\R^N} \laps{s} \tilde{u} \laps{s}\tilde{u} \\
\]
which in view of \eqref{eq:M2} establishes \eqref{eq:normconvergence}, which in turn establishes the strong $L^2(\mathbb{R}^N)$-convergence of $\laps{s} u_n$ to $\laps{s} \tilde{u}$.

Finally, let $w\in K_s(\psi)$ be arbitrary and consider the test function $v = \psi_n + w - \psi \in K_s(\psi_n)$ on the variational inequality associated to $u_n$, i.e.,
\begin{equation*}
	\int_{\R^N} \laps{s} u_n \laps{s}(\psi_n + w - \psi-u_n ) \geq f[\psi_n + w - \psi-u_n ].
\end{equation*}
Then, given that $\laps{s} u_n\to \laps{s} \tilde{u}$ in $L^2(\mathbb{R}^N)$, $\psi_n \rightharpoonup \psi$ in $H^{s,q}_{00}(\Omega)$, and $u_n \rightharpoonup \tilde{u}$ in $H^{s,2}_{00}(\Omega)$, we have 
\begin{equation*}
	\int_{\R^N} \laps{s} \tilde{u} \laps{s}( w -\tilde{u} ) \geq f[ w - \tilde{u}  ].
\end{equation*}
Since $\tilde{u}\in K_s(\psi)$, and $w\in K_s(\psi)$ was arbitrary, then $\tilde{u}=u$, the unique solution to \eqref{eq:qvi}, which concludes the proof.
\end{proof}

As in \cite{BM79} (after Definition 9) we obtain

\begin{corollary}\label{cor:Mosco}
Let $\psi_n$ weakly converge to $\psi$  in $H^{s,q}_{00}(\Omega)$ for $q>2$. Then $K_s(\psi_n)$ Mosco-converges to $K_s(\psi)$ as $n\to\infty$, i.e.,
\begin{equation*}
	K_s(\psi_n)\Mosco K_s(\psi),
\end{equation*}
\end{corollary}

\begin{proof}
We prove both items in \Cref{definition:MoscoConvergence}.
\begin{enumerate}

\item[(I)] Set $f := \laps{2s} u \in H^{s,2}_{00}(\Omega)^\ast$ where $u\in K_s(\psi)$ is arbitrary. Then
\[
 \int_{\R^N} \laps{s} u\, \laps{s} (v-u) = f[v-u] \quad \forall v \in C_c^\infty(\Omega),
\]
and in particular this equation holds for all $v \in K_s(\psi)$. For this $f$ and $\psi_n$ solve the inequality \Cref{la:existunique}, and obtain $u_n \in K_s(\psi_n)$. In view of \Cref{th1}, $u_n$ converges strongly to $u$ in $H^{s,2}_{00}(\Omega)$.

\item[(II)] follows from the Rellich compactness result: $u_n$ and $\psi_n$ converge a.e., so $u_n \geq \psi^n$ a.e. implies in the limit $u \geq \psi$ a.e., i.e.,  $u \in K_s(\psi)$.
\end{enumerate}
\end{proof}

\subsection{Quasi-variational Inequalities}
\label{s:qvi}

The above results directly allows us to establish existence of solutions to quasi-variational inequalities (QVIs) as we show next. Consider an operator $\Phi$, and subsequently the following QVI: Given $f \in H^{s,2}_{00}(\Omega)^\ast$,   
\begin{equation}\label{eq:fqvi}
	\text{Find } u \in K_s(\Phi(u)) \::\:\int_{\R^N} \laps{s} u \laps{s}( w -u ) \geq f[ w - u  ] \quad  \forall w \in K_s(\Phi(u)).
\end{equation}
We now show that the above problem has solutions via Corollary \ref {cor:Mosco}. 

\begin{theorem}
	Let $\Phi:H^{s,2}_{00}(\Omega)\to H^{s,q}_{00}(\Omega)$ for some $q>2$ be weak-weak continuous, i.e., if $v_n\rightharpoonup v$ in $H^{s,2}_{00}(\Omega)$, then $\Phi(v_n)\rightharpoonup \Phi(v)$ in $H^{s,q}_{00}(\Omega)$. Then, \eqref{eq:fqvi} has at least one solution.
\end{theorem}

\begin{proof}
	First, define the map $T$ such that $T(v)\in H^{s,2}_{00}(\Omega)$ corresponds to the unique solution of the variational inequality
\begin{equation*}
	\text{Find }\: y\in K_s(\Phi(v)) \quad:\quad  \int_{\R^N} \laps{s} y \laps{s}( w -y ) \geq f[ w - y  ], \quad \forall w \in K_s(\Phi(v)).
\end{equation*}
Then, solutions to the QVI are equivalently defined as fixed points of the map $T$, i.e., $v$ solves the QVI if and only if 
\begin{equation*}
	T(v)=v.
\end{equation*}
Coercivity of fractional differential operator implies $T(H^{s,2}_{00}(\Omega))\subset B_R(0; H^{s,2}_{00}(\Omega))$ for some $R>0$. Hence, any sequence $\{v_n\}$ in $B_R(0; H^{s,2}_{00}(\Omega))$ contains a subsequence such that $v_n\rightharpoonup v $ and  $T(v_n)\rightharpoonup z $ in $H^{s,2}_{00}(\Omega)$ for some $v$ and $z$. Since $\Phi(v_n)\rightharpoonup \Phi(v)$ in $H^{s,q}_{00}(\Omega)$ with $q>2$, then by Corollary \ref {cor:Mosco} we observe $K_s(\Phi(v_n))\Mosco K_s(\Phi(v))$, and hence $T(v_n)\to T(v) $ in $H^{s,2}_{00}(\Omega)$ by Theorem \ref{thm:Mosco}, i.e., the map $T$ is compact and a fixed point exists due to the theorem of Schauder.
\end{proof}

\section*{Remarks and open questions}\label{s:remarks}


\begin{itemize}
\item The statement of Theorem~\ref{th:convergence} is also valid for the classical Sobolev space $H^{1,p}_0(\Omega)$ for any $q < p$ -- the proof changes only slightly and is left to the interested reader. For $p=2$ this was proven in \cite{BM79,B81}. Our argument for \Cref{th:convergence} is very much inspired by \cite{B81} but it is significantly more general.
\item We expect this result to be true for functionals in $(H^{s,p}_{00}(\Omega))^\ast$ with convergence $\brac{H^{s,q}_{00}(\Omega)}^{\ast}$ for $q < p$. However, our proof does not cover this case since we are not able to classify the dual space of $H^{s,p}_{00}(\Omega)$ as in \Cref{pr:id} -- namely the global estimate \eqref{eq:id:est} is unknown when $L^2(\R^N)$ and $(H^{s,2}_{00}(\Omega))^\ast$ are replaced with $L^p(\R^N)$ and $(H^{s,p}_{00}(\Omega))^\ast$, $p \neq 2$. {The issue is boundary regularity theory for equations involving the so-called regional $s$-Laplacian, which is still under active investigation, see e.g. \cite{Fall20,Grubb15}. Observe that their results assume regularity of $\partial \Omega$.}

\item 
For $s \in (0,1)$, $p \in (1,\infty)$ denote
\[
 [f]_{W^{s,p}(\Omega)} = \brac{\int_{\Omega}\int_{\Omega} \frac{|f(x)-f(y)|^p}{|x-y|^{N+sp}}\, dx\, dy}^{\frac{1}{p}},
\]
and for $s \in (0,1)$ we set
\[
 W_{00}^{s,p}(\Omega) := \left \{ f \in L^p(\R^N) \, : \, [f]_{W^{s,p}(\R^N)}< \infty, \quad f \equiv 0 \text{ in $\Omega^c$} \right \}
\]
endowed with the norm $\|f\|_{W^{s,p}(\R^N)}$.

It is likely that a statement as in \Cref{th:convergence} fails, i.e. that there is no $(W^{s,q})^\ast$-convergence for $q <2$
-- due to the fact that $W^{s,2}$-functions may not even belong to $W^{s,p}_{loc}$ for $p<2$, \cite{MS15}.
\item Another functional analysis approach to fractional variational inequalities was presented \cite{SV13}.
\item {Results related to Murat’s theorem for non-linear operators have been investigated in \cite{BM92,DMM98}}
\end{itemize}

\appendix 

\section{On Distributions and Functions}

In this section we gather some standard results from Sobolev spaces and distributions. 
Firstly, nonnegative distributions correspond to Radon measures, more precisely we have \cite[Theorem~6.22]{LL01}. 
\begin{lemma}[Schwarz]\label{measurefact}
Every nonnegative distribution $f \in (C_c^\infty(\R^N))^\ast$, i,e.
\[
 f[\varphi] \geq 0 \quad \forall \varphi \in C_c^\infty(\R^N,[0,\infty))
\]
is of the form
\[
 f[\varphi] = \int f d\mu
\]
for some {non-negative} Radon measure $\mu$.
\end{lemma}

Secondly we recall the Vitali convergence theorem, \cite[Section 4.5, Exercise 4.14.4]{B11}.
\begin{lemma}\label{la:vitali}
Let $f_n$ be a sequence $L^p(\Omega)$, $1 \leq p < \infty$. Then $f_n \xrightarrow{n \to \infty} f$ if and only if
\begin{enumerate}
 \item convergence in measure: $\forall \eps > 0$, $\lim_{n \to \infty} \mathcal{L}^n\brac{\left \{x: |f_n(x)-f(x)|\geq \eps\right \}} = 0$
 \item for any $\eps > 0$ there exists  a measurable set $E_\eps$ with $|E_\eps|<\infty$ such that 
 \[
  \sup_{n} \int_{\Omega \backslash E_\eps} |f_n|^p < \eps.
 \]
\item for any $\eps > 0$ there exists $\delta > 0$ such that for any measurable $E \subset \Omega$ with $|E| < \delta$ we have
 \[
  \sup_{n} \int_{E} |f_n|^p < \eps.
 \]
\end{enumerate}
\end{lemma}

A direct consequence of \Cref{la:vitali} is
\begin{lemma}[Vitali convergence]\label{la:conv}
Let $\Omega \subset \R^N$ be open and bounded.
Let $f_k$ be uniformly bounded in $L^p(\Omega)$ and assume that $f_k \xrightarrow{k \to \infty} f$ in $L^1(\Omega)$. Then $f_k \xrightarrow{k \to \infty} f$ in $L^q(\Omega)$ for any $q \in [1,p)$.
\end{lemma}
\begin{proof}
The first property of Vitali follows from $L^1$-convergence.

Now observe that for $q < p$ by H\"older inequality,
\[
 \sup_{k} \int_{E} |f_k|^q \leq |E|^{\frac{p-q}{p}}\, \underbrace{\sup_{k} \|f_k\|_{L^p(\Omega)}^q}_{< \infty}
\]
This implies the third property (and the second one is trivial if $\Omega$ is bounded).
\end{proof}

Another consequence of Vitali's Lemma is the Brezis-Lieb lemma, \cite[Section 4.5, Exercise 4.17]{B11}.
\begin{lemma}[Brezis-Lieb Lemma]\label{la:BrezisLieb}
Let $p \in (1,\infty)$, $f_k,f \in L^p(\Omega)$  be such that 
\[
 \sup_{k} \|f_k\|_{L^p(\Omega)}  < \infty
\]
and $f_k \xrightarrow{k \to \infty} f$ almost everywhere in $\Omega$. Then
\begin{itemize}
\item \[
 \|f_k\|_{L^p(\Omega)}^p-\|f\|_{L^p(\Omega)}^p - \|f_k-f\|_{L^p(\Omega)}^p \xrightarrow{k \to \infty} 0.
\]
\item In particular if 
\[
 \lim_{k \to \infty} \|f_k\|_{L^p(\Omega)}^p=\|f\|_{L^p(\Omega)}^p 
\]
then $f_k \xrightarrow{k \to \infty} f$ in $L^p(\Omega)$.
\end{itemize}
\end{lemma}

\section*{Acknowledgment}
Financial support is acknowledged as follows:
\begin{itemize}
 \item HA is partially supported by Air Force Office of Scientific Research under Award NO: FA9550-19-1-0036, 
NSF grants DMS-1818772 and DMS-1913004, and Department of the Navy, Naval Postgraduate School under 
Award NO: N00244-20-1-0005.
\item CNR acknowledges the support of Germany's Excellence Strategy - The Berlin Mathematics Research Center MATH+ (EXC-2046/1, project ID: 390685689) within project AA4-3.
\item AS is supported by Simons foundation, grant no 579261.
\end{itemize}
The authors are grateful to the anonymous referees for careful reading and suggestions.

\bibliographystyle{abbrv}
\bibliography{bib}

\end{document}